\documentclass[reqno,11pt]{article}%
\usepackage{amsmath}
\usepackage{graphicx}
\usepackage{amsfonts}
\usepackage{amssymb}%
 \makeatletter
\newtheorem{theorem}{Theorem}

\newtheorem{corollary}[theorem]{Corollary}

\newtheorem{lemma}[theorem]{Lemma}

\newtheorem{proposition}[theorem]{Proposition}

\newenvironment{proof}[1][Proof]{\noindent{\textbf {#1}  }}  {\hfill$\Box$}

\begin{document}

\title{An analytic theory of extremal hypergraph problems\thanks{\textbf{AMS MSC:}
05C65; 05C35. \textbf{Keywords:} \textit{uniform hypergraphs; extremal
problems; hereditary property; largest eigenvalue; graph Lagrangians.}}}
\author{Vladimir Nikiforov\thanks{Department of Mathematical Sciences, University of
Memphis, Memphis TN 38152, USA; email: \textit{vnikifrv@memphis.edu}}}
\maketitle

\begin{abstract}
The starting point of this paper is the following problem: If $\mathcal{P}$ is
a hereditary property of $r$-uniform graphs, find the limit%
\[
\pi\left(  \mathcal{P}\right)  =\lim_{n\rightarrow\infty}\binom{n}{r}^{-1}%
\max\{e\left(  G\right)  :\text{ }G\in\mathcal{P}\text{ and }v\left(
G\right)  =n\}.\text{ \ \ }%
\]
Is is shown that this problem is just a particular case of a general analytic
problem about a parameter $\lambda^{\left(  \alpha\right)  }\left(  G\right)
$ defined for every $r$-graph $G$ and every real $\alpha\geq1$ as
\[
\lambda^{\left(  \alpha\right)  }\left(  G\right)  =\max_{\left\vert
x_{1}\right\vert ^{\alpha}\text{ }+\text{ }\left\vert x_{2}\right\vert
^{\alpha}\text{ }+\text{ }\cdots\text{ }+\text{ }\left\vert x_{n}\right\vert
^{\alpha}\text{ }=\text{ }1}r!\sum_{\{i_{1},i_{2},\ldots i_{r}\}\in E\left(
G\right)  }x_{i_{1}}x_{i_{2}}\cdots x_{i_{r}},
\]
Note that $\lambda^{\left(  1\right)  }\left(  G\right)  $ is a well-studied
parameter, however, the truly exceptional value is $\lambda\left(  G\right)
=\lambda^{\left(  r\right)  }\left(  G\right)  ,$ known as the largest
eigenvalue of $G.$

Two of the main results of the paper are: for all $\alpha\geq1\ $the limit%
\[
\lambda^{\left(  \alpha\right)  }\left(  \mathcal{P}\right)  =\lim
_{n\rightarrow\infty}n^{r/\alpha-r}\max\{\lambda^{\left(  \alpha\right)
}\left(  G\right)  :\text{ }G\in\mathcal{P}\text{ and }v\left(  G\right)
=n\}\text{ \ \ }%
\]
exists, and if $\alpha>1,$ then%
\[
\lambda^{(\alpha)}\left(  \mathcal{P}\right)  =\pi\left(  \mathcal{P}\right)
.
\]

It is shown also that if $\lambda^{(1)}\left(  \mathcal{P}\right)  =\pi\left(
\mathcal{P}\right)  ,$ then $\mathcal{P}$ has remarkable features regarding
extremal problems. Many known concrete results are generalized and further
research is outlined.

\end{abstract}

\section{Introduction and main results}

In this paper we build an analytic theory of hypergraph extremal problems of
the type:\medskip

\emph{Given an }$r$-\emph{uniform} \emph{graph }$G$ \emph{of order }$n,$
\emph{belonging to some} \emph{hereditary property }$\mathcal{P}\emph{,}$
\emph{how many edges can }$G$ \emph{have?}\medskip

One of the key results proved in this paper states that this combinatorial
problem is asymptotically equivalent to an extremal analytic problem: \medskip

\emph{If }$G$ \emph{is as above,} \emph{how large the largest eigenvalue of
}$G$\emph{ can be?} \medskip

To clarify this point, let us lay down the definition of the largest
eigenvalue $\lambda\left(  G\right)  $ that is adopted here:\medskip

\emph{Suppose that the vertices of }$G$\emph{ are the integers }$1,2,\ldots
,n$\emph{ and let }$x_{1},x_{2},\ldots,x_{n}$\emph{ be real numbers. Set}
\begin{equation}
\lambda\left(  G\right)  =\max_{\left\vert x_{1}\right\vert ^{r}%
+\cdots+\left\vert x_{n}\right\vert ^{r}=1}r!\sum\left\{  x_{i_{1}}x_{i_{2}%
}\cdots x_{i_{r}}:\left\{  i_{1},i_{2},\ldots,i_{r}\right\}  \text{ is an edge
of }G\right\}  . \label{def}%
\end{equation}

The value $\lambda\left(  G\right)  $ turns out to be at the meeting point of
two major lines of research - one is on maxima of homogenous polynomial forms
on graphs, and the other is on spectra of hypermatrices. We start by
highlighting a few milestones along these two lines, some of which are all but forgotten.

\subsection{Background}

The study of polynomial forms on graphs and their maxima over unit spheres in
the $l^{1}$ norm has been pioneered by Motzkin and Straus in \cite{MoSt65},
and later generalized by Khadziivanov \cite{Kha77} and by S\'{o}s and Straus
\cite{SoSt82}, see also \cite{NikA} for some historical remarks. For
hypergraphs the same topic has been studied first by Brown and Simonovits
\cite{BrSi84}, and later by Sidorenko \cite{Sid87}, with some very definite
results; other similar early contributions are by Frankl and R\"{o}dl
\cite{FrRo84} and Frankl and F\"{u}redi \cite{FrFu88}. While for $2$-graphs
this method has been enlightening, for hypergraphs the obtained results were
less encompassing, due to the fact that this research remained bound to unit
spheres in the $l^{1}$ norm, as was the original result of Motzkin and Straus.

On the other hand, for even positive $r,$ the study of critical points of
polynomial forms of degree $r$ over finite dimensional unit spheres in the
$l^{r}$ norm has been suggested by Lusternik and Schnirelman already in 1930,
see \cite{LuSh30}, p. 38, or its French translation \cite{LuSh34}. This topic
has been developed further by Krasnoselskii \cite{Kra55}, Elsholtz,
Tzitlanadze, and others, but the focal point of these later contributions has
shifted to infinite dimensional spaces. Nevertheless, the underlying idea of
Lusternik and Schnirelman, forgotten for decades, nowadays became mainstream,
under the name of \textquotedblleft variational eigenvalues\textquotedblright%
\ of hypermatrices. Indeed, in the same spirit, recently Lim \cite{Lim05}
proposed a variational approach to spectra of hypermatrices of both even and
odd dimensions. Independently, an algebraic approach to the same goal was
proposed by Qi in \cite{Qi05}. For further developments on spectra of
hypermatrices see \cite{CPZ08},\cite{CPZ09},\cite{FGH11},\cite{Qi06}%
,\cite{Qi07},\cite{YaYa10},\cite{YaYa11},\cite{YaYa13}.

\subsection{Overview}

One point that needs clarification is why the maximum in (\ref{def}) is taken
over the unit sphere in the $l^{r}$ norm. First, most of the definitions of
eigenvalues adopted in the above cited papers on spectra of hypermatrices
reduce precisely to (\ref{def}) for the largest eigenvalue of hypergraphs.
Second, as shown by Friedland, Gauber and Han \cite{FGH11}, the use of the
$l^{r}$ norm is indeed necessary, if we want to preserve some essential
features of the Perron-Frobenius theory for $r$-dimensional matrices.

While the largest eigenvalue $\lambda\left(  G\right)  $ is exceptional, in
this paper we study a more general parameter $\lambda^{\left(  \alpha\right)
}\left(  G\right)  ,$ defined exactly as $\lambda\left(  G\right)  $ in
(\ref{def}), but with maximum taken over the unit sphere in the $l^{\alpha}$
norm, where $\alpha\geq1$ is a real number. For $\alpha>1$ the parameter
$\lambda^{\left(  \alpha\right)  }\left(  G\right)  $ has been introduced and
used by Keevash, Lenz and Mubayi in \cite{KLM13}, but they \ provided just
scanty groundwork on $\lambda^{\left(  \alpha\right)  }\left(  G\right)  $.
Thus, one of the goals of this paper is to set a more solid base for a
systematic study of $\lambda^{\left(  \alpha\right)  }\left(  G\right)  $ and
of $\lambda\left(  G\right)  .$

Before continuing let us stress that $\lambda^{\left(  \alpha\right)  }\left(
G\right)  $ is defined as a conditional maximum; thus, its usability in
extremal problems is rooted in its very nature. Indeed, for $2$-graphs it has
been shown that many classical results can be enhanced and recast for the
largest eigenvalue $\lambda\left(  G\right)  $ with an astonishing
preservation of detail; for a survey of these results see \cite{Nik11}, and
for some new developments also \cite{NikC}. Unfortunately, for hypergraphs the
present situation is not so advanced: there are just a few isolated spectral
extremal results, mainly due to Cooper and Dutle \cite{CoDu11} and to Keevash,
Lenz and Mubayi \cite{KLM13}.

Moreover, $\lambda^{\left(  \alpha\right)  }\left(  G\right)  $ is just one of
the many critical values that can be defined in a setting similar to
(\ref{def}), and at least some of them are applicable to extremal graph
problems. For example, these possibilities have been explored for the smallest
and the second largest eigenvalue of $2$-graphs. In this paper we focus
exclusively on $\lambda^{\left(  \alpha\right)  }\left(  G\right)  ,$ with a
brief discussion of other similar parameters in the concluding remarks.

As already said, we shall show that edge extremal problems are asymptotically
equivalent to extremal problems for the largest eigenvalue. In this way, known
extremal edge results readily imply spectral bounds for hypergraphs. However,
the opposite implication seems more significant, as it paves the road for an
extensive use of differential calculus in hypergraphs. For example, finding
the maximum possible eigenvalue of a graph with a forbidden subgraph gives
asymptotically the maximum possible number of its edges, yet solving the
former problem might be easier, using known analytical techniques.

We also illustrate the use of various proof tools in solving concrete
problems, in particular, in problems for flat and multiplicative properties,
to be defined later. These tools include the Hypergraph Removal Lemma,
classical inequalities, Lagrange multipliers and other methods from real
analysis. The concrete results that we obtain shed new light on several older
results of Sidorenko \cite{Sid87}, and on some new ones by Keevash, Lenz and
Mubayi \cite{KLM13}.

The paper ends up with a summary discussion and open problems.

\subsection{The basics}

Recall that an $r$\textbf{-uniform hypergraph (}$r$\textbf{-graph)} consists
of a set of vertices $V\left(  G\right)  $ and a set $E\left(  G\right)  $ of
$r$-subsets of $V,$ called edges. We set $v\left(  G\right)  =\left\vert
V\right\vert $ and $e(G)=\left\vert E\left(  G\right)  \right\vert .$ When
$V\left(  G\right)  $ is not defined explicitly, it is assumed that $V\left(
G\right)  =[n]=\left\{  1,2,,\ldots,n\right\}  .$

Given an $r$-graph $G$ and a vector $\mathbf{x}=\left(  x_{1},x_{2}%
,\ldots,x_{n}\right)  \in\mathbb{R}^{n},$ the \textbf{polyform} of $G$ is the
function $P_{G}\left(  \mathbf{x}\right)  $ defined as%
\[
P_{G}\left(  \mathbf{x}\right)  =r!\sum_{\left\{  i_{1},i_{2},\ldots
,i_{r}\right\}  \in E}x_{i_{1}}x_{i_{2}}\cdots x_{i_{r}}.
\]
Note that $P_{G}\left(  \mathbf{x}\right)  $ is a homogenous polynomial of
degree $r$ and is linear in each variable $x_{i}.$ Clearly, the definition
(\ref{def}) is equivalent to
\[
\lambda\left(  G\right)  =\max_{\left\Vert \mathbf{x}\right\Vert _{r}=1}%
P_{G}\left(  \mathbf{x}\right)  .
\]

The largest eigenvalue $\lambda\left(  G\right)  $ has turned out to be a
versatile parameter, with close relations to many properties of $G;$ see
\cite{NikB} for some results. As already said, the choice of the $l^{r}$ norm
in the definition of $\lambda\left(  G\right)  $ makes it exceptional, but it
will be useful to consider a more general parameter $\lambda^{\left(
\alpha\right)  }\left(  G\right)  $, defined for every real number $\alpha
\geq1$ as
\[
\lambda^{\left(  \alpha\right)  }\left(  G\right)  =\max_{\left\Vert
\mathbf{x}\right\Vert _{\alpha}=1}P_{G}\left(  \mathbf{x}\right)
=\max_{\left\vert x_{1}\right\vert ^{\alpha}+\cdots+\left\vert x_{n}%
\right\vert ^{\alpha}=1}r!\sum_{\left\{  i_{1},i_{2},\ldots,i_{r}\right\}  \in
E\left(  G\right)  }x_{i_{1}}x_{i_{2}}\cdots x_{i_{r}}.
\]
Note first that $\lambda^{\left(  r\right)  }\left(  G\right)  =\lambda\left(
G\right)  ,$ and second, that $\lambda^{\left(  1\right)  }\left(  G\right)  $
is another much studied parameter, known as the Lagrangian\footnote{Let us
note that this use of the name \emph{Lagrangian} is at odds with the
tradition. Indeed, names as \emph{Laplacian, Hessian, Gramian, Grassmanian},
etc., usually denote a structured object like matrix, operator, or manifold,
and not just a single number.} of $G$. So $\lambda^{\left(  \alpha\right)
}\left(  G\right)  $ can link $\lambda\left(  G\right)  $ to a large body of
previous work on extremal hypergraph problems.

The purpose of the following propositions is twofold: first, to give the
reader some insight in the meaning and use of $\lambda^{\left(  \alpha\right)
}\left(  G\right)  ;$ second, these general results, together with the results
in Section \ref{prol}, set the background for more thorough future study of
$\lambda^{\left(  \alpha\right)  }\left(  G\right)  .$ On more than one
occasion we shall see the special role of the case $\alpha=r.$\medskip

First, taking the $n$-vector $\mathbf{x}=\left(  n^{-1/\alpha},\ldots
,n^{-1/\alpha}\right)  ,$ we immediately get
\[
\lambda^{\left(  \alpha\right)  }\left(  G\right)  \geq P_{G}\left(
\mathbf{x}\right)  =r!e\left(  G\right)  /n^{r/\alpha}.
\]
On the other hand, for $\alpha>1,$ Keevash, Lenz and Mubayi have proved that
$\lambda^{\left(  \alpha\right)  }\left(  G\right)  \leq\left(  r!e\left(
G\right)  \right)  ^{1-1/\alpha}$ in \cite{KLM13}, Lemma 5. We shall improve
this result in Theorem \ref{th8} and Corollary \ref{cor3} below, which also
allow for some additional fine-tuning. Here is the summary of these bounds.

\begin{proposition}
\label{pro1}If $\alpha\geq1$ and $G$ is an $r$-graph of order $n$, then
\begin{equation}
r!e\left(  G\right)  /n^{r/\alpha}\leq\lambda^{\left(  \alpha\right)  }\left(
G\right)  \leq\left(  r!e\left(  G\right)  \right)  ^{1-1/\alpha}. \label{gin}%
\end{equation}
If $G$ contains at least one edge, then $\lambda^{\left(  \alpha\right)
}\left(  G\right)  <\left(  r!e\left(  G\right)  \right)  ^{1-1/\alpha}.$
\end{proposition}

Inequalities (\ref{gin}) show that $\lambda^{\left(  \alpha\right)  }\left(
G\right)  $ tends to $r!e\left(  G\right)  $ when $\alpha\rightarrow\infty.$
Noting that
\[
\lambda^{\left(  \alpha\right)  }\left(  G\right)  =\max_{\left\vert
y_{1}\right\vert +\cdots+\left\vert y_{n}\right\vert =1}r!\sum_{\left\{
i_{1},i_{2},\ldots,i_{r}\right\}  \in E\left(  G\right)  }\left\vert y_{i_{1}%
}\right\vert ^{1/\alpha}\left\vert y_{i_{2}}\right\vert ^{1/\alpha}%
\cdots\left\vert y_{i_{r}}\right\vert ^{1/\alpha}.
\]
it becomes also clear that $\lambda^{\left(  \alpha\right)  }\left(  G\right)
$ is increasing and continuous in $\alpha.$

\begin{proposition}
\label{pro2}If $\alpha\geq1$ and $G$ is an $r$-graph, then $\lambda^{\left(
\alpha\right)  }\left(  G\right)  $ is increasing and continuous in $\alpha$,
and
\[
\lim_{\alpha\rightarrow\infty}\lambda^{\left(  \alpha\right)  }\left(
G\right)  =r!e\left(  G\right)  .
\]

\end{proposition}

A cornerstone bound on $\lambda\left(  G\right)  $ for a $2$-graph $G,$ with
maximum degree $\Delta,$ is the inequality $\lambda\left(  G\right)
\leq\Delta.$ For $r$-graphs this generalizes to $\lambda\left(  G\right)
\leq\left(  r-1\right)  !\Delta,$ but if $1\leq\alpha<r$,\ there is no
analogous bound for $\lambda^{\left(  \alpha\right)  }\left(  G\right)  ,$
which would be tight. Here is what we can say presently on these relations.

\begin{proposition}
\label{pro3}Let $G$ be an $r$-graph of order $n,$ with maximum degree
$\Delta.$

(i) If $\alpha\geq r,$ then
\begin{equation}
\lambda^{\left(  \alpha\right)  }\left(  G\right)  \leq\frac{\left(
r-1\right)  !\Delta}{n^{r/\alpha-1}}, \label{inmax}%
\end{equation}
with equality holding if and only if $G$ is regular;

(ii) If $1\leq\alpha<r,$ there exist $r$-graphs $G$ for which (\ref{inmax}) fails;

(iii) If $1\leq\alpha<r,$ then $\lambda^{\left(  \alpha\right)  }\left(
G\right)  <\left(  r-1\right)  !\Delta.$
\end{proposition}

Another cornerstone result about $\lambda\left(  G\right)  $ of a $2$-graph
$G$ of order $n$ is: $\lambda\left(  G\right)  =2e\left(  G\right)  /n$\emph{
if and only if }$G$\emph{ is regular.} We saw in Proposition \ref{pro1} that
the inequality $\lambda\left(  G\right)  \geq2e\left(  G\right)  /n$
generalizes seamlessly for $\lambda^{\left(  \alpha\right)  }\left(  G\right)
$ of any $r$-graph $G$ and any $\alpha\geq1$, but as shown below the condition
for equality becomes quite intricate, even for $r=2.$

\begin{proposition}
\label{pro4}If $\alpha\geq1$ and $\lambda^{\left(  \alpha\right)  }\left(
G\right)  =r!e\left(  G\right)  /n^{r/\alpha},$ then $G$ is regular. If
$\alpha\geq r$ and $G$ is regular, then $\lambda^{\left(  \alpha\right)
}\left(  G\right)  =r!e\left(  G\right)  /n^{r/\alpha}.$ However, if
$1\leq\alpha<r,$\ then there exist regular graphs $G$ such that $\lambda
^{\left(  \alpha\right)  }\left(  G\right)  >r!e\left(  G\right)
/n^{r/\alpha}.$
\end{proposition}

Proposition \ref{pro2} states that $\lambda^{\left(  \alpha\right)  }\left(
G\right)  $ increases in $\alpha$. Here are two useful technical statements
which give some information how fast $\lambda^{\left(  \alpha\right)  }\left(
G\right)  $ can increase indeed.

\begin{proposition}
\label{pro5}If $\alpha\geq1$ and $G$ is an $r$-graph, then the function%
\[
h_{G}\left(  \alpha\right)  =\lambda^{\left(  \alpha\right)  }\left(
G\right)  n^{r/\alpha}%
\]
is nonincreasing in $\alpha.$ If $G$ is non-regular, then $h_{G}\left(
\alpha\right)  $ is decreasing in $\alpha.$
\end{proposition}

\begin{proposition}
\label{pro6}If $\alpha\geq1$ and $G$ is an $r$-graph, then the function
\[
f_{G}\left(  \alpha\right)  =\left(  \frac{\lambda^{\left(  \alpha\right)
}\left(  G\right)  }{r!e\left(  G\right)  }\right)  ^{\alpha}%
\]
is nonincreasing in $\alpha.$
\end{proposition}

\subsection{Graph properties and asymptotics of extremal problems}

In this paper extremal graph problems are discussed in the general setting of
\textbf{properties} of $r$-graphs, which are just families of $r$-graphs
closed under isomorphisms. Given a property $\mathcal{P},$ we shall write
$\mathcal{P}_{n}$ for the set of all graphs in $\mathcal{P}$ of order $n$. A
property is called \textbf{monotone }if it is closed under taking subgraphs,
and \textbf{hereditary},\textbf{ }if it is closed under taking induced
subgraphs. For example, given a set of $r$-graphs $\mathcal{F},$ the family of
all $r$-graphs that do not contain any $F\in\mathcal{F}$ as a subgraph is a
monotone property, denoted by $Mon\left(  \mathcal{F}\right)  .$ Likewise, the
family of all $r$-graphs that do not contain any $F\in\mathcal{F}$ as an
induced subgraph is a hereditary property, denoted as $Her\left(
\mathcal{F}\right)  .$ When $\mathcal{F}$ consists of a single graph $F,$ we
shall write $Mon\left(  F\right)  $\ and $Her\left(  F\right)  $ instead of
$Mon\left(  \left\{  F\right\}  \right)  $ and $Her\left(  \left\{  F\right\}
\right)  .$

The extremal problems studied below stem from the following one: \emph{Given a
hereditary property }$\mathcal{P}$\emph{ of }$r$\emph{-graphs, find}
\begin{equation}
ex\left(  \mathcal{P},n\right)  =\max_{G\in\mathcal{P}_{n}}e\left(  G\right)
.\label{exdef}%
\end{equation}
If $r=2$ and $\mathcal{P}$ is a monotone property, sharp asymptotics of
$ex\left(  \mathcal{P},n\right)  $ is known$,$ but general hereditary
properties seem to have been shrugged off, although a simple and appealing
asymptotic solution also exists, see \cite{NikC} for details. For $r\geq3$ the
problem has turned out to be generally very hard and has been solved only for
very few properties $\mathcal{P};$ see \cite{Kee11} for an up-to-date
discussion. An easier asymptotic version of the same problem arises from the
following fact, established by Katona, Nemetz and Simonovits \cite{KNS64}.

\begin{proposition}
\label{proKNS}If $\mathcal{P}$ is a hereditary property of $r$-graphs, then
the sequence
\[
\left\{  ex\left(  \mathcal{P},n\right)  \binom{n}{r}^{-1}\right\}
_{n=1}^{\infty}%
\]
is nonincreasing and so the limit
\[
\pi\left(  \mathcal{P}\right)  =\lim_{n\rightarrow\infty}ex\left(
\mathcal{P},n\right)  \binom{n}{r}^{-1}%
\]
always exists.
\end{proposition}

Thus, if we find $\pi\left(  \mathcal{P}\right)  ,$ we can obtain $ex\left(
\mathcal{P},n\right)  $ asymptotically, but even $\pi\left(  \mathcal{P}%
\right)  $ is hard to find for most properties $\mathcal{P}$, in particular,
$\pi\left(  Mon\left(  F\right)  \right)  $ is not known for many simple
graphs $F.$

As it turns out, the parameters $\lambda^{\left(  \alpha\right)  }\left(
G\right)  $, and in particular $\lambda\left(  G\right)  ,$ can be efficient
tools for the study of $\pi\left(  \mathcal{P}\right)  .$ Indeed, given a
hereditary property $P$ of $r$-graphs, set in analogy to (\ref{exdef})%
\[
\lambda^{\left(  \alpha\right)  }\left(  \mathcal{P},n\right)  =\max
_{G\in\mathcal{P}_{n}}\lambda^{\left(  \alpha\right)  }\left(  G\right)  .
\]
Now choosing $G\in\mathcal{P}_{n}$ with maximum number of edges, Proposition
\ref{pro1} implies that
\[
\lambda^{\left(  \alpha\right)  }\left(  G\right)  \geq r!ex\left(
\mathcal{P},n\right)  /n^{r/\alpha},
\]
and so
\begin{equation}
\lambda^{\left(  \alpha\right)  }\left(  \mathcal{P},n\right)  \geq
r!ex\left(  \mathcal{P},n\right)  /n^{r/\alpha}. \label{inled}%
\end{equation}

Let us begin with a theorem about $\lambda^{\left(  \alpha\right)  }\left(
G\right)  $, which is similar to Proposition \ref{proKNS}.

\begin{theorem}
\label{th1}Let $\alpha\geq1.$ If $\mathcal{P}$ is a hereditary property of
$r$-graphs, then the limit
\begin{equation}
\lambda^{\left(  \alpha\right)  }\left(  \mathcal{P}\right)  =\lim
_{n\rightarrow\infty}\lambda^{\left(  \alpha\right)  }\left(  \mathcal{P}%
,n\right)  n^{r/\alpha-r}\label{exlima}%
\end{equation}
exists. If $\alpha=1,$ then $\lambda^{\left(  1\right)  }\left(
\mathcal{P},n\right)  $ is nondecreasing, and so
\begin{equation}
\lambda^{\left(  1\right)  }\left(  \mathcal{P},n\right)  \leq\lambda^{\left(
1\right)  }\left(  \mathcal{P}\right)  .\label{bnd1}%
\end{equation}
If $\alpha>1,$ then $\lambda^{\left(  \alpha\right)  }\left(  \mathcal{P}%
\right)  $ satisfies%
\begin{equation}
\lambda^{\left(  \alpha\right)  }\left(  \mathcal{P}\right)  \leq\frac
{\lambda^{\left(  \alpha\right)  }\left(  \mathcal{P},n\right)  n^{r/\alpha
-1}}{\left(  n-1\right)  \left(  n-2\right)  \ldots\left(  n-r+1\right)
}.\label{bndsa}%
\end{equation}

\end{theorem}

Interestingly, Theorem \ref{th1} is as important as Proposition \ref{proKNS},
and its proof is not too hard either, yet it seems to have been missed even in
the much studied case $\alpha=1$.

Here is an immediate consequence of Theorem \ref{th1}. From (\ref{inled}) we
see that
\[
\frac{\lambda^{\left(  \alpha\right)  }\left(  \mathcal{P},n\right)
n^{\left(  r/\alpha\right)  -1}}{\left(  n-1\right)  \left(  n-2\right)
\cdots\left(  n-r+1\right)  }\geq\frac{ex\left(  \mathcal{P},n\right)
}{\binom{n}{r}},
\]
and letting $n\rightarrow\infty,$ we find also that
\begin{equation}
\lambda^{\left(  \alpha\right)  }\left(  \mathcal{P}\right)  \geq\pi\left(
\mathcal{P}\right)  .\label{limgin}%
\end{equation}
We shall show that, almost always, equality holds in this inequality.

An important property of $\lambda^{\left(  \alpha\right)  }\left(
\mathcal{P}\right)  $ is that it is nonincreasing in $\alpha.$ Note the
difference with Proposition \ref{pro2}, which states that $\lambda^{\left(
\alpha\right)  }\left(  G\right)  $ is increasing in $\alpha$ for every fixed
graph $G.$

\begin{theorem}
\label{th2}If $\mathcal{P}$ is a hereditary property of $r$-graphs, then
$\lambda^{\left(  \alpha\right)  }\left(  \mathcal{P}\right)  $ is
nonincreasing in $\alpha\geq1$.
\end{theorem}

We deduce Theorem \ref{th2} from the following subtler relation, which itself
is obtained from Proposition \ref{pro6}.

\begin{proposition}
\label{pro7}If $\mathcal{P}$ is a hereditary property of $r$-graphs, and
$1\leq\alpha\leq\beta,$ then%
\begin{equation}
\left(  \frac{\lambda^{\left(  \beta\right)  }\left(  \mathcal{P}\right)
}{\pi\left(  \mathcal{P}\right)  }\right)  ^{\beta}\leq\left(  \frac
{\lambda^{\left(  \alpha\right)  }\left(  \mathcal{P}\right)  }{\pi\left(
\mathcal{P}\right)  }\right)  ^{\alpha}. \label{inabp}%
\end{equation}

\end{proposition}

\bigskip

Before concluding this subsection, we shall give an immediate application of
Theorem \ref{th1}, but since it refers to blow-up of an $r$-graph, let us
first we define this concept:\medskip

\emph{Given an }$r$\emph{-graph }$H$\emph{ of order }$h$\emph{ and positive
integers }$k_{1},\ldots,k_{n}$\emph{, write }$H\left(  k_{1},\ldots
,k_{h}\right)  $\emph{ for the graph obtained by replacing each vertex }$v\in
V\left(  H\right)  $\emph{ with a set }$U_{v}$\emph{ of size }$x_{v}$\emph{
and each edge }$\left\{  v_{1},\ldots,v_{r}\right\}  \in E\left(  H\right)
$\emph{ with a complete }$r$\emph{-partite }$r$\emph{-graph with vertex
classes }$U_{v_{1}},\ldots,U_{v_{r}}.$\emph{ The graph }$H\left(  k_{1}%
,\ldots,k_{h}\right)  $\emph{ is called a} \textbf{blow-up} \emph{of}
$H.$\medskip

It is well known (see, e.g., \cite{Kee11}, Theorem 2.2) that if $H$ is an
$r$-graph of order $h$ and $H\left(  k_{1},\ldots,k_{h}\right)  $ is a fixed
blow-up of $H,$ then%
\begin{equation}
\pi\left(  Mon\left(  H\right)  \right)  =\pi\left(  Mon\left(  H\left(
k_{1},\ldots,k_{h}\right)  \right)  \right)  . \label{edblo}%
\end{equation}

It turns out that a similar result holds for $\lambda^{\left(  \alpha\right)
}:$

\begin{theorem}
\label{thblo}If $\alpha>1,$ $H$ is an $r$-graph of order $h$ and $H\left(
k_{1},\ldots,k_{h}\right)  $ is a fixed blow-up of $H,$ then
\[
\lambda^{\left(  \alpha\right)  }\left(  Mon\left(  H\right)  \right)
=\lambda^{\left(  \alpha\right)  }\left(  Mon\left(  H\left(  k_{1}%
,\ldots,k_{h}\right)  \right)  \right)  .
\]

\end{theorem}

Our proof of Theorem \ref{thblo} is not long, but it is based on the
Hypergraph Removal Lemma and other fundamental results about $r$-graphs. Note
also that there are simple examples showing that the theorem does not hold for
$\alpha=1.$

\subsection{The equivalence of $\lambda^{\left(  \alpha\right)  }\left(
\mathcal{P}\right)  $ and $\pi\left(  \mathcal{P}\right)  $}

Since $\pi\left(  \mathcal{P}\right)  $ and $\lambda^{\left(  \alpha\right)
}\left(  \mathcal{P}\right)  $ are defined alike, one anticipates some close
relation between them to hold. For example, for any $r$-graph $G$ of order
$n,$ inequality (\ref{gin}) implies that
\begin{equation}
\lambda\left(  G\right)  \geq r!e\left(  G\right)  /n.\label{in1}%
\end{equation}
From (\ref{limgin}) we get also that if $\mathcal{P}$ is a hereditary
property, then
\begin{equation}
\lambda\left(  \mathcal{P}\right)  \geq\pi\left(  \mathcal{P}\right)
.\label{in1.1}%
\end{equation}
\qquad Now, if we know that $\lambda\left(  \mathcal{P}\right)  =\pi\left(
\mathcal{P}\right)  ,$ then for every $\alpha>r,$ inequality (\ref{limgin})
and Theorem \ref{th2} imply that
\[
\pi\left(  \mathcal{P}\right)  \leq\lambda^{\left(  \alpha\right)  }\left(
\mathcal{P}\right)  \leq\lambda\left(  \mathcal{P}\right)  =\pi\left(
\mathcal{P}\right)  ,
\]
and so, $\lambda^{\left(  a\right)  }\left(  \mathcal{P}\right)  =\pi\left(
\mathcal{P}\right)  $ as well$.$

It turns out that equality always holds in (\ref{in1.1}), as stated in the
following theorem, which is the central result of the paper:

\begin{theorem}
\label{th3}If $\mathcal{P}$ is a hereditary property of $r$-graphs, then for
every $\alpha>1,$
\begin{equation}
\lambda^{\left(  a\right)  }\left(  \mathcal{P}\right)  =\pi\left(
\mathcal{P}\right)  . \label{maineq}%
\end{equation}

\end{theorem}

It seems that a result of this scope is not available in the literature, even
for $2$-graphs, so some remarks are due here. First, using (\ref{maineq}),
every result about $\pi\left(  \mathcal{P}\right)  $ of a hereditary property
$\mathcal{P}$ gives a result about $\lambda\left(  \mathcal{P}\right)  $ as
well, so we readily obtain a number of results for the largest eigenvalue of
uniform graphs. But equality (\ref{maineq}) is more significant, as the left
and right hand sides of (\ref{in1}) could be quite different. Moreover,
finding $\pi\left(  \mathcal{P}\right)  $ now can be reduced to maximization
of a smooth function subject to a smooth constraint, and in this kind of
problem Lagrange multipliers can provide much information on the structure of
the extremal graphs. For example, such approach was used successfully in
\cite{Nik11a}.

We showed that for $\alpha>r$ inequality (\ref{maineq}) follows immediately
from the case $\alpha=r,$ but our proof of the case $1<\alpha\leq r$ is not
easy. The proof, given in Section \ref{pf}, is quite technical and mostly
analytic. It is based on several lemmas of their own interest, which are
presented in Section \ref{le}.

Moreover, for $r=2$ the value of $\pi\left(  \mathcal{P}\right)  $ can be
characterized explicitly and so one can establish $\lambda^{\left(  a\right)
}\left(  \mathcal{P}\right)  $ as well for all $\alpha>1.$ We give here a
short proof for monotone properties, referring the reader to \cite{NikC} for
general hereditary properties.

\begin{theorem}
\label{thCon2}Let $\alpha>1.$ If $\mathcal{P}$ is a monotone property of
$2$-graphs, then
\[
\lambda^{\left(  a\right)  }\left(  \mathcal{P}\right)  =\frac{r-2}{r-1},
\]
where $r=\min\left\{  \chi\left(  G\right)  :G\notin\mathcal{P}\right\}  $.
\end{theorem}

It is not difficult to find a hereditary property $\mathcal{P}$ of $2$-graphs
for which $\lambda^{\left(  1\right)  }\left(  \mathcal{P}\right)  >\pi\left(
\mathcal{P}\right)  .$ Indeed, let $K_{3}\left(  1,2,2\right)  $ be the
blow-up of a triangle and let $\mathcal{P}=Mon\left(  K_{3}\left(
1,2,2\right)  \right)  $. First, Theorem \ref{thblo} and Motzkin-Straus's
result imply that for every $\alpha>1,$%
\[
\lambda^{\left(  \alpha\right)  }\left(  Mon\left(  K_{3}\left(  1,2,2\right)
\right)  \right)  =\lambda^{\left(  \alpha\right)  }\left(  Mon\left(
K_{3}\right)  \right)  =\lambda^{\left(  1\right)  }\left(  Mon\left(
K_{3}\right)  \right)  =1/2.
\]
Now, taking $G_{n}$ to be the graph consisting of a $K_{4}$ and $n-4$ isolated
vertices, we see that $G_{n}\in\mathcal{P}_{n}$ for $n\geq4.$ But
$\lambda^{\left(  1\right)  }\left(  G_{n}\right)  =3/4,$ and so
$\lambda^{\left(  1\right)  }\left(  \mathcal{P}\right)  \geq3/4>1/2.$

Obviously, Theorem \ref{th3} puts in focus hereditary properties $\mathcal{P}$
for which $\lambda^{\left(  1\right)  }\left(  \mathcal{P}\right)  $ also
satisfies (\ref{maineq}); thus the definition: \emph{a hereditary property
}$P$\emph{ of }$r$\emph{-graphs is called }\textbf{flat}\emph{ if }%
$\lambda^{\left(  1\right)  }\left(  \mathcal{P}\right)  =\pi\left(
\mathcal{P}\right)  .\medskip$

It turns out that flat properties possess truly remarkable features with
respect to extremal problems, some of which are presented in the next subsection.

\subsection{Flat properties}

Let us note that, in general, $\pi\left(  \mathcal{P}\right)  $ alone is not
sufficient to estimate $ex\left(  \mathcal{P},n\right)  $ for small values of
$n$ and for arbitrary hereditary property $\mathcal{P}.$ However, flat
properties allow for tight, explicit upper bounds on $ex\left(  \mathcal{P}%
,n\right)  $ and $\lambda^{\left(  \alpha\right)  }\left(  \mathcal{P}\right)
$. To emphasize the substance of the general statements in this subsection, we
first outline a class of flat properties, whose study has been started by
Sidorenko \cite{Sid87}, albeit in a different setting.\medskip

\emph{A graph property }$P$\emph{ is said to be multiplicative if }$G\in
P_{n}$\emph{ implies that }$G\left(  k_{1},\ldots,k_{n}\right)  \in P$\emph{
\ for every vector of positive integers }$k_{1},\ldots,k_{n}.$ \emph{This is
to say, a multiplicative property contains the blow-ups of all its
members.}\medskip

Multiplicative properties are quite convenient for extremal graph theory, and
they are ubiquitous as well. Indeed, following Sidorenko \cite{Sid87}, call a
graph $F$ \textbf{covering} if every two vertices of $F$ are contained in an
edge. Clearly, complete $r$-graphs are covering, and for $r=2$ these are the
only covering graphs, but for $r\geq3$ there are many noncomplete ones. For
example, the Fano plane $3$-graph $F_{7}$ is a noncomplete covering graph.
Obviously, if $F$ is a covering graph, then $Mon\left(  F\right)  $ is both a
hereditary and a multiplicative property.

Below we illustrate Theorems \ref{th4} and \ref{th5} using $F_{7}$ as
forbidden graph because it is covering and $\pi\left(  Mon\left(
F_{7}\right)  \right)  $ is known. To keep our presentation focused, we stick
to $F_{7}$ only, but there are other graphs with the same properties; for
instance, Keevash in \cite{Kee11}, Sec. 14, lists several such graphs, like
\textquotedblleft expanded triangle\textquotedblright, \textquotedblleft%
$3$-book with $3$ pages\textquotedblright, \textquotedblleft$4$-book with $4$
pages\textquotedblright\ and others. Using these and similar references, the
reader may easily come up with other illustrations. Let us point that these
applications are new and are not available in the literature.

Two other examples of hereditary, multiplicative properties are based on
vertex colorings. Recall that the \textbf{chromatic number} $\chi\left(
G\right)  $ of an $r$-graph $G\ $is the smallest number $\chi$ such that
$V\left(  G\right)  $ can be partitioned into $\chi$ edgeless sets. Likewise,
the \textbf{weak} \textbf{chromatic number} $\underline{\chi}\left(  G\right)
$ is the smallest number $\chi$ such that $V\left(  G\right)  $ can be
partitioned into $\chi$ sets so that every set intersects every edge in at
most one vertex.

Let now $\mathcal{C}\left(  p\right)  $ be the family of all $r$-graphs $G$
with $\chi\left(  G\right)  \leq p$ and $\underline{\mathcal{C}}\left(
q\right)  $ be the family of all $r$-graphs $G$ with $\underline{\chi}\left(
G\right)  \leq p.$ Note first that $\mathcal{C}\left(  p\right)  $ and
$\underline{\mathcal{C}}\left(  q\right)  $ are hereditary and multiplicative
properties, so they are also flat. This statement is more or less obvious, but
it does not follow by forbidding covering subgraphs.

The following proposition summarizes the principal facts about $\mathcal{C}%
\left(  p\right)  $ and $\underline{\mathcal{C}}\left(  q\right)  $.

\begin{proposition}
\label{pro9} For all $p,q$ the classes $\mathcal{C}\left(  p\right)  $ and
$\underline{\mathcal{C}}\left(  q\right)  $ are hereditary and multiplicative
properties, and
\[
\pi\left(  \mathcal{C}\left(  p\right)  \right)  =\left(  1-p^{-r+1}\right)
\ \ \ \text{and}\ \ \ \ \pi\left(  \underline{\mathcal{C}}\left(  q\right)
\right)  =r!\binom{q}{r}q^{-r}.
\]

\end{proposition}

The motivation for the next result comes from Sidorenko \cite{Sid87}, Theorem
2.6, who proved that if $F$ is a covering graph, then
\begin{equation}
\lambda^{\left(  1\right)  }\left(  Mon\left(  F\right)  \right)  =\pi\left(
Mon\left(  F\right)  \right)  .\label{Seq}%
\end{equation}
This can be recast in our terminology as: \emph{if }$F$\emph{ is a covering
graph, then }$Mon\left(  F\right)  $\emph{ is a flat property}. To analyze the
underpinnings of this result, note that $Mon\left(  F\right)  $ is both
hereditary and multiplicative property. The following theorem gives a natural
generalization of (\ref{Seq}).

\begin{theorem}
\label{th4}If $\mathcal{P}$ is a hereditary, multiplicative property, then it
is flat; that is to say,
\begin{equation}
\lambda^{\left(  \alpha\right)  }\left(  \mathcal{P}\right)  =\pi\left(
\mathcal{P}\right)  \label{eqa}%
\end{equation}
for every $\alpha\geq1.$
\end{theorem}

To illustrate the usability of Theorem \ref{th4} note that the $3$-graph
$F_{7}$ is covering, and, as determined in \cite{FuSi05} and \cite{KeSu05},
$\pi\left(  Mon\left(  F_{7}\right)  \right)  =3/4,$ so we immediately get
that if $\alpha\geq1,$ then
\[
\lambda^{\left(  \alpha\right)  }\left(  Mon\left(  F_{7}\right)  \right)
=3/2.
\]
However, below we show that even more convenient bounds are available in this
and similar cases. Indeed a distinctive feature of all flat properties, and
the one that justifies the introduction of the concept, is the fact that there
exist neat and tight upper bounds on $\lambda^{\left(  \alpha\right)  }\left(
G\right)  $ and $e\left(  G\right)  $ for every graph $G$ that belongs to a
flat property. This claim is substantiated in Theorems \ref{th5} and \ref{th6} below.

\begin{theorem}
\label{th5}If $\mathcal{P}$ is a flat property, and $G\in\mathcal{P}_{n},$
then
\begin{equation}
e\left(  G\right)  \leq\pi\left(  \mathcal{P}\right)  n^{r}/r!. \label{Seq1}%
\end{equation}
and for every $\alpha\geq1,$%
\begin{equation}
\lambda^{\left(  \alpha\right)  }\left(  G\right)  \leq\pi\left(
\mathcal{P}\right)  n^{r-r/\alpha}. \label{upb}%
\end{equation}
Both inequalities (\ref{Seq1}) and (\ref{upb}) are tight.
\end{theorem}

When $\mathcal{P}=Mon\left(  F\right)  $ and $F$ is a covering graph, the
bound (\ref{Seq1}) has been proved by Sidorenko in \cite{Sid87}, Theorem 2.3.
Clearly, Theorem \ref{th5} is much more general, although its proof is similar
to that of Sidorenko. Taking again the Fano plane as an example, we obtain the
following tight inequality:

\begin{corollary}
\label{cor2} If $G$ is a $3$-graph of order $n$, not containing the Fano
plane, then for all $\alpha\geq1,$
\begin{equation}
\lambda^{\left(  \alpha\right)  }\left(  G\right)  \leq\frac{3}{4}%
n^{3-3/\alpha}. \label{fpin}%
\end{equation}

\end{corollary}

This inequality is essentially equivalent to Corollary 3 in \cite{KLM13},
albeit it is somewhat less precise. We believe however, that Theorem \ref{th5}
shows clearly why such a result is possible at all.

With respect to chromatic number, an early result of Cvetkovi\'{c}
\cite{Cve72} states: \emph{if }$G$\emph{ is a }$2$\emph{-graph of order }%
$n$\emph{ and chromatic number }$\chi,$\emph{ then}
\[
\lambda\left(  G\right)  \leq\frac{\chi-1}{\chi}n.
\]
This bound easily generalizes for hypergraphs.

\begin{corollary}
Let $G$ be an $r$-graph of order $n$ and let $\alpha\geq1.$

(i) If $\chi\left(  G\right)  =\chi$, then
\[
\lambda^{\left(  \alpha\right)  }\left(  G\right)  \leq\left(  1-\chi
^{-r+1}\right)  n^{r-r/\alpha};
\]

(ii) If $\underline{\chi}\left(  G\right)  =\underline{\chi}$, then
\[
\lambda^{\left(  \alpha\right)  }\left(  G\right)  \leq r!\binom
{\underline{\chi}}{r}\underline{\chi}^{-r}n^{r-r/\alpha}.
\]

\end{corollary}

Furthermore, recalling that complete graphs are the only covering $2$-graphs,
it becomes clear that the bound (\ref{upb}) is analogous to Wilf's bound
\cite{Wil86}: \emph{if }$G$\emph{ is a }$2$\emph{-graph of order }$n$\emph{
and clique number }$\omega,$\emph{ then}%
\begin{equation}
\lambda\left(  G\right)  \leq\frac{\omega-1}{\omega}n. \label{wilin}%
\end{equation}

Inequality (\ref{wilin}) has been improved by a subtler inequality in
\cite{Nik02}, namely: \emph{if }$G$\emph{ is a }$2$\emph{-graph with }%
$m$\emph{ edges and clique number }$\omega,$\emph{ then}%
\begin{equation}
\lambda\left(  G\right)  \leq\sqrt{\frac{2\left(  \omega-1\right)  }{\omega}%
m}. \label{edin}%
\end{equation}
To see that (\ref{edin}) implies (\ref{wilin}) it is enough to recall the
Tur\'{a}n bound $m\leq\left(  1-1/\omega\right)  n^{2}/2$. It turns out that
the proof of (\ref{edin}) generalizes to hypergraphs, giving the following
theorem, which strengthens (\ref{upb}) exactly as (\ref{edin}) strengthens
(\ref{wilin}).

\begin{theorem}
\label{th6}If $\mathcal{P}$ is a flat property, and $G\in\mathcal{P}$, then
\begin{equation}
\lambda^{\left(  \alpha\right)  }\left(  G\right)  \leq\pi\left(  G\right)
^{1/\alpha}\left(  r!e\left(  G\right)  \right)  ^{1-1/\alpha}. \label{turin}%
\end{equation}

\end{theorem}

Let us emphasize the peculiar fact that the bound (\ref{turin}) does not
depend on the order of $G,$ but it is asymptotically tight in many cases. In
particular, for $3$-graphs with no $F_{7}$ we obtain the following tight bound:

\begin{corollary}
If $G$ is a $3$-graph with $m$ edges, and $G$ does not contain the Fano plane,
then
\[
\lambda^{\left(  \alpha\right)  }\left(  G\right)  \leq3\cdot2^{1-3/\alpha
}m^{1-1/\alpha}.
\]

\end{corollary}

Finally, we shall use (\ref{turin}) to improve the inequality
\[
\lambda^{\left(  \alpha\right)  }\left(  G\right)  \leq\left(  r!e\left(
G\right)  \right)  ^{1-1/\alpha},
\]
given in Lemma 5 of \cite{KLM13}. First, note that Proposition \ref{pro9},
together with Theorem \ref{th6}, gives the following general bounds:

\begin{theorem}
\label{th8}If $G$ is an $r$-graph and $\alpha\geq1,$ then
\[
\lambda^{\left(  \alpha\right)  }\left(  G\right)  \leq\left(  1-\chi\left(
G\right)  ^{-r+1}\right)  ^{1/\alpha}\left(  r!e\left(  G\right)  \right)
^{1-1/\alpha}.
\]
and
\[
\lambda^{\left(  \alpha\right)  }\left(  G\right)  \leq r!\binom
{\underline{\chi}\left(  G\right)  }{r}^{1/\alpha}\underline{\chi}\left(
G\right)  ^{-r/\alpha}e\left(  G\right)  ^{1-1/\alpha}.
\]

\end{theorem}

In particular, in view of $\chi\left(  G\right)  \leq v\left(  G\right)
/\left(  r-1\right)  $ and $\underline{\chi}\left(  G\right)  \leq v\left(
G\right)  ,$ we obtain simple bounds in terms of the order and size:

\begin{corollary}
\label{cor3}If $\alpha\geq1$ and $G$ is an $r$-graph of order $n,$ then
\[
\lambda^{\left(  \alpha\right)  }\left(  G\right)  \leq\left(  1-\left(
\frac{n}{r-1}\right)  ^{-r+1}\right)  ^{1/\alpha}\left(  r!e\left(  G\right)
\right)  ^{1-1/\alpha}.
\]
and%
\[
\lambda^{\left(  \alpha\right)  }\left(  G\right)  \leq r!\binom{n}%
{r}^{1/\alpha}n^{-r/\alpha}e\left(  G\right)  ^{1-1/\alpha}%
\]
If $G$ contains at least one edge, then
\[
\lambda^{\left(  \alpha\right)  }\left(  G\right)  <\left(  r!e\left(
G\right)  \right)  ^{1-1/\alpha}.
\]

\end{corollary}

\medskip

The above results lead to the natural question: \emph{Is there a flat property
that is not multiplicative?} The answer is yes, there exists a flat property
$\mathcal{P}$ of $2$-graphs that is not multiplicative. Indeed, let
$\mathcal{P}=Her\left(  C_{4}\right)  ,$ that is to say, $\mathcal{P}$ is the
class of all graphs with no induced $4$-cycle. Trivially, all complete graphs
belong to $\mathcal{P}$ and so%
\[
\lambda^{\left(  \alpha\right)  }\left(  \mathcal{P}\right)  =\pi\left(
G\right)  =1\text{ \ \ for all }\alpha\geq1.
\]
However, obviously $\mathcal{P}$ is not multiplicative, as $C_{4}=K_{2}\left(
2,2\right)  $.\medskip

Analyzing the above example, we come up with the following sufficient
condition for flat properties.

\begin{theorem}
\label{th9}Let $\mathcal{F}$ be a set of $r$-graphs each of which is a blow-up
of a covering graph. Then $Her\left(  \mathcal{F}\right)  $ is flat.
\end{theorem}

Apparently Theorem \ref{th9} greatly extends the range of flat properties,
however further work is needed to determine the limits of its
applicability.\medskip

The rest of the paper is organized as follows: in Section \ref{prol} we give
general results for the parameters $\lambda^{\left(  \alpha\right)  }\left(
G\right)  $ and in particular for $\lambda\left(  G\right)  .$ In Section
\ref{le} we prove two useful lemmas which are needed in the proof of Theorem
\ref{th4}, but are of separate interest as well. The proofs of the various
statements above are presented in Section \ref{pf}. The paper ends with a
summary discussion and open problems in Section \ref{CR}.

\section{\label{prol}Some properties of $\lambda\left(  G\right)  $ and
$\lambda^{\left(  \alpha\right)  }\left(  G\right)  $}

Below we write $\binom{X}{k}$ for the set of all $k$-subsets of a set $X.$ Let
us recall also the notation $\left(  n\right)  _{r}$ for the falling factorial
$n!/\left(  n-k\right)  !$%
\[
\left(  n\right)  _{r}=\frac{n!}{\left(  n-k\right)  !}=n\left(  n-1\right)
\cdots\left(  n-k+1\right)  .
\]
In our profs we shall use extensively Jensen's and Maclauren's inequalities;
the reader is referred to \cite{HLP88} for ground material.

Most of the basic results about $\lambda^{\left(  \alpha\right)  }\left(
G\right)  $ appear here for the first time, and we hope that they will be
useful to other researchers.

Let $G$ be an $r$-graph of order $n.$ Since its polyform $P_{G}\left(
\mathbf{x}\right)  $ is homogenous, we find that
\begin{equation}
\lambda\left(  G\right)  =\max_{\left\Vert \mathbf{x}\right\Vert _{r}=1}\text{
}P_{G}\left(  \mathbf{x}\right)  =\max_{\mathbf{x}\neq0}\frac{P_{G}\left(
\mathbf{x}\right)  }{\left\Vert \mathbf{x}\right\Vert _{r}^{r}}=\max
_{\mathbf{x}\neq0}P_{G}\left(  \frac{1}{\left\Vert \mathbf{x}\right\Vert _{r}%
}\mathbf{x}\right)  . \label{eiv}%
\end{equation}
We shall call a nonzero real vector $\mathbf{x}=\left(  x_{1},\ldots
,x_{n}\right)  $ an \textbf{eigenvector }to $\lambda\left(  G\right)  $ if
\[
\lambda\left(  G\right)  =P_{G}\left(  \frac{1}{\left\Vert \mathbf{x}%
\right\Vert _{r}}\mathbf{x}\right)  .
\]
Note that relation (\ref{eiv}) holds for $\lambda\left(  G\right)  ,$ but it
is not true for $\lambda^{\left(  \alpha\right)  }\left(  G\right)  $ if
$\alpha\neq r.$ This fact corroborates again the exclusivity of $\lambda
\left(  G\right)  $. However, let us note a useful inequality for the general
$\lambda^{\left(  \alpha\right)  }\left(  G\right)  ,$ which we shall use
later with no explicit reference.

\begin{proposition}
Let $\alpha\geq1.$ If $G$ is an $r$-graph of order $n,$ and $\mathbf{x}%
=\left(  x_{1},\ldots,x_{n}\right)  $ is any real vector, then
\[
P_{G}\left(  \mathbf{x}\right)  \leq\lambda^{\left(  \alpha\right)  }\left(
G\right)  \left\Vert \mathbf{x}\right\Vert _{\alpha}^{r}.
\]

\end{proposition}

Here are two other obvious facts:

\begin{proposition}
If $G$ is an $r$-graph with at least one edge then $\lambda^{\left(
\alpha\right)  }\left(  G\right)  \geq r!r^{-r/\alpha}.$ If $H$ is a subgraph
of $G,$ then $\lambda^{\left(  \alpha\right)  }\left(  H\right)  \leq
\lambda^{\left(  \alpha\right)  }\left(  G\right)  $.
\end{proposition}

We state below a handy fact that will be used later with no explicit reference.

\begin{proposition}
For any $\alpha\geq1$ there is a nonnegative vector $\mathbf{x}$ such that
$\left\Vert \mathbf{x}\right\Vert _{\alpha}=1$ and $\lambda^{\left(
\alpha\right)  }\left(  G\right)  =P_{G}\left(  \mathbf{x}\right)  .$
\end{proposition}

Indeed, if $\mathbf{x}=\left(  x_{1},\ldots,x_{n}\right)  $ is a vector such
that $\left\Vert \mathbf{x}\right\Vert _{\alpha}=1$ and%
\begin{equation}
\lambda^{\left(  \alpha\right)  }\left(  G\right)  =P_{G}\left(
\mathbf{x}\right)  =\max_{\left\Vert \mathbf{x}\right\Vert _{\alpha}=1}%
P_{G}\left(  \mathbf{x}\right)  , \label{aeq}%
\end{equation}
then $\lambda^{\left(  \alpha\right)  }\left(  G\right)  =P_{G}\left(
\left\vert x_{1}\right\vert ,\ldots,\left\vert x_{n}\right\vert \right)
.\bigskip$

If $\alpha=r,$ there are stronger statements, which are analogous to
statements in the Perron-Frobenius theory for nonnegative matrices. The
following crucial statement can be deduced from the results of Friedland,
Gauber and Han \cite{FGH11}, but an independent, direct proof has been given
also by Cooper and Dutle in \cite{CoDu11}.

\begin{theorem}
\label{SRHGth}Let $G$ be a connected $r$-graph. Then $\lambda\left(  A\right)
$ has a positive eigenvector, which is unique up to scaling.$\hfill\square$
\end{theorem}

Note, however, that for $r\geq3,$ even if $G$ is a connected $r$-graph,
$\lambda\left(  G\right)  $ may have eigenvectors with negative and positive
entries; for example, if $G$ is a one edge graph, then the $r$-vector $\left(
1,1,\ldots,1\right)  $ is an eigenvector to $\lambda\left(  G\right)  ,$ but
so is the vector $\left(  -1,-1,1,\ldots,1\right)  $ as well.

Let now $\alpha>1$ and let $\mathbf{x}=\left(  x_{1},\ldots,x_{n}\right)  $ be
a nonnegative vector satisfying $\left\Vert \mathbf{x}\right\Vert _{\alpha}=1$
and (\ref{aeq}). Using Lagrange multipliers, we find that there exists $\mu$
such that, for every $k=1,\ldots,n,$
\[
\mu\alpha x_{k}^{\alpha-1}=\left(  r-1\right)  !\sum_{\left\{  k,i_{1}%
,\ldots,i_{r-1}\right\}  \in E\left(  G\right)  }x_{i_{1}}\cdots x_{i_{r-1}}.
\]
Multiplying the $k$'th equation by $x_{k}$ and adding them all, we find that
$\mu\alpha=\lambda^{\left(  \alpha\right)  }\left(  G\right)  $, and so, the
numbers $x_{1},\ldots,x_{n}$ satisfy the $n$ equations
\begin{equation}
\lambda^{\left(  \alpha\right)  }\left(  G\right)  x_{k}^{\alpha-1}=\left(
r-1\right)  !\sum_{\left\{  k,i_{1},\ldots,i_{r-1}\right\}  \in E\left(
G\right)  }x_{i_{1}}\cdots x_{i_{r-1}},\text{ \ \ \ \ }1\leq k\leq n.
\label{eequ}%
\end{equation}
For $\alpha>1$ these equations are a powerful tool in the study of
$\lambda^{\left(  \alpha\right)  }\left(  G\right)  $ and particularly of
$\lambda\left(  G\right)  ,$ but they are not always available for $\alpha=1.$

Next, we shall prove a relation between $\lambda^{\left(  \alpha\right)
}\left(  G\right)  $ and $\lambda^{\left(  \alpha\right)  }\left(  G\left(
k,\ldots,k\right)  \right)  ,$ where $G\left(  k,\ldots,k\right)  $ is a
uniform blow-up of $G.$

\begin{proposition}
\label{pro8}If $G$ is an $r$-graph and $k\geq1$ is an integer, then
\[
\lambda^{\left(  \alpha\right)  }\left(  G\left(  k,\ldots,k\right)  \right)
=k^{r-r/\alpha}\lambda^{\left(  \alpha\right)  }\left(  G\right)  .
\]

\end{proposition}

\begin{proof}
By definition, the vertex set $V\left(  G\left(  k,\ldots,k\right)  \right)  $
of $G\left(  k,\ldots,k\right)  $ can be partitioned into $n$ disjoint sets
$U_{1},\ldots,U_{n}$ each consisting of $k$ vertices. Also, if $\left\{
i_{1},\ldots,i_{r}\right\}  \in E\left(  G\right)  ,$ then $\left\{
j_{1},\ldots,j_{r}\right\}  \in$ $E\left(  G\left(  k,\ldots,k\right)
\right)  $ for every $j_{1}\in U_{i_{1}},j_{2}\in U_{i_{2}},\ldots,j_{r}\in
U_{i_{r}}.$ We shall prove first that
\begin{equation}
\lambda^{\left(  \alpha\right)  }\left(  G\left(  k,\ldots,k\right)  \right)
\geq k^{r-r/\alpha}\lambda^{\left(  \alpha\right)  }\left(  G\right)
.\label{in6}%
\end{equation}
Let $\mathbf{x}=\left(  x_{1},\ldots,x_{n}\right)  $ be a nonnegative vector
such that $\left\Vert \mathbf{x}\right\Vert _{\alpha}=1$ and \ $\lambda
^{\left(  \alpha\right)  }\left(  G\right)  =P_{G}\left(  \mathbf{x}\right)
.$ For every $i\in\left[  n\right]  $ and every $j\in U_{i},$ set
\[
y_{j}=\frac{1}{k^{1/\alpha}}x_{i}%
\]
The vector $\mathbf{y}=\left(  y_{1},\ldots,y_{nk}\right)  $ satisfies
$\left\Vert \mathbf{y}\right\Vert _{\alpha}=1,$ and therefore,%
\[
\lambda^{\left(  \alpha\right)  }\left(  G\left(  k,\ldots,k\right)  \right)
\geq P_{G\left(  k,\ldots,k\right)  }\left(  \mathbf{y}\right)  =\frac
{1}{k^{r/\alpha}}k^{r}P_{G}\left(  \mathbf{x}\right)  =k^{r-r/\alpha}%
\lambda^{\left(  \alpha\right)  }\left(  G\right)  ,
\]
proving (\ref{in6}). To complete the proof of the proposition, we shall show
that%
\begin{equation}
\lambda^{\left(  \alpha\right)  }\left(  G\left(  k,\ldots,k\right)  \right)
\leq k^{r-r/\alpha}\lambda^{\left(  \alpha\right)  }\left(  G\right)
.\label{in5}%
\end{equation}
Let $\mathbf{x}=\left(  x_{1},\ldots,x_{nk}\right)  $ be a nonnegative
$kn$-vector such that $\left\Vert \mathbf{x}\right\Vert _{\alpha}=1$ and
\[
\lambda^{\left(  \alpha\right)  }\left(  G\left(  k,\ldots,k\right)  \right)
=P_{G\left(  k,\ldots,k\right)  }\left(  \mathbf{x}\right)  .
\]
By definition,%
\begin{align*}
P_{G\left(  k,\ldots,k\right)  }\left(  \mathbf{x}\right)   &  =r!\sum
_{\left\{  i_{1},\ldots,i_{r}\right\}  \in E\left(  G\left(  k,\ldots
,k\right)  \right)  }x_{i_{1}}x_{i_{2}}\cdots x_{i_{r}}\\
&  =r!\sum_{\left\{  i_{1},\ldots,i_{r}\right\}  \in E\left(  G\right)
}\left(  \sum_{j\in U_{i_{1}}}x_{j}\right)  \left(  \sum_{j\in U_{i_{2}}}%
x_{j}\right)  \cdots\left(  \sum_{j\in U_{i_{r}}}x_{j}\right)
\end{align*}
Next, for every $s\in\left[  n\right]  ,$ using Jensen's inequality, we see
that%
\[
\sum_{j\in U_{s}}x_{j}\leq k^{1-1/\alpha}\left(  \sum_{j\in U_{s}}%
x_{j}^{\alpha}\right)  ^{1/\alpha}.
\]
Now, setting for every $s\in\left[  n\right]  ,$%
\[
y_{s}=\left(  \sum_{j\in U_{s}}x_{j}^{\alpha}\right)  ^{1/\alpha},
\]
we find a vector $\mathbf{y}=\left(  y_{1},\ldots,y_{n}\right)  $ with
$\left\Vert \mathbf{y}\right\Vert _{\alpha}=1.$ Also,
\begin{align*}
\lambda^{\left(  \alpha\right)  }\left(  G\left(  k,\ldots,k\right)  \right)
&  =P_{G\left(  k,\ldots,k\right)  }\left(  \mathbf{x}\right)  \\
&  =r!\sum_{\left\{  i_{1},\ldots,i_{r}\right\}  \in E\left(  G\right)
}\left(  \sum_{j\in U_{i_{1}}}x_{j}\right)  \left(  \sum_{j\in U_{i_{2}}}%
x_{j}\right)  \cdots\left(  \sum_{j\in U_{i_{r}}}x_{j}\right)  \\
&  \leq r!k^{r-r/\alpha}\sum_{\left\{  i_{1},\ldots,i_{r}\right\}  \in
E\left(  G\right)  }y_{i_{1}y_{i_{2}}}\ldots y_{i_{r}}=k^{r-r/\alpha}%
P_{G}\left(  \mathbf{y}\right)  \\
&  \leq k^{r-r/\alpha}\lambda^{\left(  \alpha\right)  }\left(  G\right)  .
\end{align*}
This completes the proof of (\ref{in5}), and with (\ref{in6}), also the proof
of Proposition \ref{pro8}.
\end{proof}

\medskip

Finally, we give a perturbation bound on $\lambda^{\left(  \alpha\right)
}\left(  G\right)  ,$ which is used to estimate how much changes
$\lambda^{\left(  \alpha\right)  }\left(  G\right)  ,$ when edges of $G$ are changed.

\begin{proposition}
\label{pro10}Let $\alpha\leq1,$ $k\geq1$ and $G_{1}$ and $G_{2}$ be $r$-graphs
on the same vertex set. If $G_{1}$ and $G_{2}$ differ in at most $k$ edges,
then%
\[
\left\vert \lambda^{\left(  \alpha\right)  }\left(  G_{1}\right)
-\lambda^{\left(  \alpha\right)  }\left(  G_{2}\right)  \right\vert
\leq\left(  r!k\right)  ^{1-1/\alpha}.
\]

\end{proposition}

\begin{proof}
Let $V=V\left(  G_{1}\right)  =V\left(  G_{2}\right)  $ and write $G_{12}$ for
the $r$-graph with $V\left(  G_{12}\right)  =V$ and $E\left(  G_{12}\right)
=E\left(  G_{1}\right)  \cap E\left(  G_{2}\right)  .$ We may and shall assume
that $\lambda^{\left(  \alpha\right)  }\left(  G_{1}\right)  \geq
\lambda^{\left(  \alpha\right)  }\left(  G_{2}\right)  .$ Write $G_{3}$ for
the $r$-graph with $V\left(  G_{3}\right)  =V$ and $E\left(  G_{3}\right)
=E\left(  G_{1}\right)  \backslash E\left(  G_{2}\right)  .$ In view of
$G_{12}\subset G_{2},$ we have%
\begin{align*}
0  &  \leq\lambda^{\left(  \alpha\right)  }\left(  G_{1}\right)
-\lambda^{\left(  \alpha\right)  }\left(  G_{2}\right)  =\lambda^{\left(
\alpha\right)  }\left(  G_{1}\right)  -\lambda^{\left(  \alpha\right)
}\left(  G_{12}\right)  -\left(  \lambda^{\left(  \alpha\right)  }\left(
G_{2}\right)  -\lambda^{\left(  \alpha\right)  }\left(  G_{12}\right)  \right)
\\
&  \leq\lambda^{\left(  \alpha\right)  }\left(  G_{1}\right)  -\lambda
^{\left(  \alpha\right)  }\left(  G_{12}\right)  =\lambda^{\left(
\alpha\right)  }\left(  G_{3}\right)  \leq\left(  r!e\left(  G_{3}\right)
\right)  ^{1-1/\alpha}\\
&  \leq\left(  r!k\right)  ^{1-1/\alpha},
\end{align*}
proving Proposition \ref{pro10}.
\end{proof}

\section{\label{le}Two lemmas about critical vectors to $\lambda^{\left(
\alpha\right)  }\left(  G\right)  $}

A useful result in spectral extremal theory for $2$-graphs is the following
bound from \cite{Nik09}:

\emph{Let }$G$\emph{ be an }$2$\emph{-graph with minimum degree }$\delta
,$\emph{ and let }$x=\left(  x_{1},\ldots,x_{n}\right)  $\emph{ be a
nonnegative eigenvector to }$\lambda\left(  G\right)  $\emph{ with
}$\left\Vert \mathbf{x}\right\Vert _{2}=1.$\emph{ Then the value }%
$x=\min\left\{  x_{1},\ldots,x_{n}\right\}  $\emph{ satisfies}%
\begin{equation}
x^{2}\left(  \lambda\left(  G\right)  ^{2}+\delta n-\delta^{2}\right)
\leq\delta\label{minx}%
\end{equation}

The bound (\ref{minx}) is exact for many different graphs, and as explained in
\cite{Nik11}, it has been crucial in proving upper bounds on $\lambda\left(
G\right)  $ by induction on the number of vertices of $G.$ Very likely, a
similar bound for hypergraphs would be useful as well. Below we state and
prove such a result; despite its awkward form, for $r=\alpha=2$ it yields
precisely (\ref{minx}); moreover, it is crucial for the proof of Theorem
\ref{th3}.

\begin{lemma}
\label{le1}Let $1\leq\alpha\leq r,$ and let $G$ be an $r$-graph of order $n,$
with minimum degree $\delta$ and with $\lambda^{\left(  \alpha\right)
}\left(  G\right)  =\lambda.$ Let $\mathbf{x}=\left(  x_{1},\ldots
,x_{n}\right)  $ be a nonnegative vector such that $\left\Vert \mathbf{x}%
\right\Vert _{\alpha}=1$ and $\lambda=P_{G}\left(  \mathbf{x}\right)  .$ Then
the value $x=\min\left\{  x_{1},\ldots,x_{n}\right\}  $ satisfies%
\[
\left(  \left(  \frac{\lambda n^{r/\alpha-1}}{\left(  r-1\right)  !}\right)
^{\alpha}-\delta^{\alpha}\right)  x^{\alpha\left(  r-1\right)  }\leq
\binom{n-1}{r-1}\delta^{\alpha-1}\left(  \frac{\left(  1-x^{\alpha}\right)
^{r-1}}{\left(  n-1\right)  ^{r-1}}-x^{\alpha\left(  r-1\right)  }\right)  .
\]

\end{lemma}

\begin{proof}
Set for short $V=V\left(  G\right)  $ and let $k\in V$ be a vertex of degree
$\delta$. The equation (\ref{eequ}) for the vertex $k$ implies that%
\[
\lambda x^{\alpha-1}\leq\lambda x_{k}^{\alpha-1}=\left(  r-1\right)
!\sum_{\left\{  k,i_{1},\ldots,i_{r-1}\right\}  \in E\left(  G\right)
}x_{i_{1}}\ldots x_{i_{r-1}}.
\]
Now, dividing by $\left(  r-1\right)  !$ and applying Jensen's inequality to
the right-hand side, we find that
\begin{equation}
\left(  \frac{\lambda x^{\alpha-1}}{\left(  r-1\right)  !}\right)  ^{\alpha
}\leq\delta^{\alpha-1}\sum_{\left\{  k,i_{1},\ldots,i_{r-1}\right\}  \in
E\left(  G\right)  }x_{i_{1}}^{\alpha}\ldots x_{i_{r-1}}^{\alpha}.\label{in4}%
\end{equation}
Our next goal is to bound the quantity $\sum_{\left\{  k,i_{1},\ldots
,i_{r-1}\right\}  \in E\left(  G\right)  }x_{i_{1}}^{r}\ldots x_{i_{r-1}}^{r}$
from above$.$ First, note that
\begin{align}
&  \sum_{\left\{  k,i_{1},\ldots,i_{r-1}\right\}  \in E\left(  G\right)
}x_{i_{1}}^{\alpha}\ldots x_{i_{r-1}}^{\alpha}\nonumber\\
&  =\sum_{\left\{  i_{1},\ldots,i_{r-1}\right\}  \in\binom{V\backslash\left\{
v_{k}\right\}  }{r-1}}x_{i_{1}}^{\alpha}\ldots x_{i_{r-1}}^{\alpha}-\left(
\sum_{\left\{  i_{1},\ldots,i_{r-1}\right\}  \in\binom{V\backslash\left\{
v_{k}\right\}  }{r-1}\text{ and }\left\{  k,i_{1},\ldots,i_{r-1}\right\}
\notin E\left(  G\right)  }x_{i_{1}}^{\alpha}\ldots x_{i_{r-1}}^{\alpha
}\right)  \nonumber\\
&  \leq\sum_{\left\{  i_{1},\ldots,i_{r-1}\right\}  \in\binom{V\backslash
\left\{  v_{k}\right\}  }{r-1}}x_{i_{1}}^{\alpha}\ldots x_{i_{r-1}}^{\alpha
}-\left(  \sum_{\left\{  i_{1},\ldots,i_{r-1}\right\}  \in\binom
{V\backslash\left\{  v_{k}\right\}  }{r-1}\text{ and }\left\{  k,i_{1}%
,\ldots,i_{r-1}\right\}  \notin E\left(  G\right)  }x^{\alpha\left(
r-1\right)  }\right)  \nonumber\\
&  =\sum_{\left\{  i_{1},\ldots,i_{r-1}\right\}  \in\binom{V\backslash\left\{
v_{k}\right\}  }{r-1}}x_{i_{1}}^{\alpha}\ldots x_{i_{r-1}}^{\alpha}-\left(
\binom{n-1}{r-1}-\delta\right)  x^{\alpha\left(  r-1\right)  }.\label{in7}%
\end{align}
Next, applying Maclauren's inequality for the $\left(  r-1\right)  $'th
symmetric function of the variables $x_{i}^{\alpha},$ $i\in V\backslash
\left\{  v_{k}\right\}  $, we find that%
\begin{align*}
\frac{1}{\binom{n-1}{r-1}}\sum_{\left\{  i_{1},\ldots,i_{r-1}\right\}
\in\binom{V\backslash\left\{  v_{k}\right\}  }{r-1}}x_{i_{1}}^{\alpha}\ldots
x_{i_{r-1}}^{\alpha} &  \leq\left(  \frac{1}{n-1}\sum_{i\in V\backslash
\left\{  k\right\}  }x_{i}^{\alpha}\right)  ^{r-1}=\frac{1}{\left(
n-1\right)  ^{r-1}}\left(  1-x_{k}^{\alpha}\right)  ^{r-1}\\
&  \leq\frac{1}{\left(  n-1\right)  ^{r-1}}\left(  1-x^{\alpha}\right)
^{r-1}.
\end{align*}
Hence, replacing in (\ref{in7}), we obtain the desired bound
\[
\sum_{\left\{  k,i_{1},\ldots,i_{r-1}\right\}  \in E\left(  G\right)
}x_{i_{1}}^{\alpha}\ldots x_{i_{r-1}}^{\alpha}\leq\frac{\binom{n-1}{r-1}%
}{\left(  n-1\right)  ^{r-1}}\left(  1-x^{\alpha}\right)  ^{r-1}-\left(
\binom{n-1}{r-1}-\delta\right)  x^{\alpha\left(  r-1\right)  }.
\]
Returning back to (\ref{in4}), we see that
\[
\left(  \frac{\lambda}{\left(  r-1\right)  !}\right)  ^{\alpha}x^{\left(
\alpha-1\right)  \alpha}\leq\binom{n-1}{r-1}\delta^{\alpha-1}\left(
\frac{\left(  1-x^{\alpha}\right)  ^{r-1}}{\left(  n-1\right)  ^{r-1}%
}-x^{\alpha\left(  r-1\right)  }\right)  +\delta^{\alpha}x^{\alpha\left(
r-1\right)  }.
\]

Since $\alpha\leq r$ and $x\leq n^{-1/\alpha},$ we see that $x^{\alpha-r}\geq
n^{-\left(  \alpha-r\right)  /\alpha}=n^{r/\alpha-1}$ and therefore,%
\[
\frac{\lambda n^{r/\alpha-1}}{\left(  r-1\right)  !}x^{r-1}\leq\frac{\lambda
}{\left(  r-1\right)  !}x^{\alpha-1}.
\]
Hence,%
\begin{align*}
\left(  \frac{\lambda n^{r/\alpha-1}}{\left(  r-1\right)  !}\right)  ^{\alpha
}x^{\alpha\left(  r-1\right)  }  & \leq\left(  \frac{\lambda}{\left(
r-1\right)  !}\right)  ^{\alpha}x^{\left(  \alpha-1\right)  \alpha}\\
& \leq\binom{n-1}{r-1}\delta^{\alpha-1}\left(  \frac{\left(  1-x^{\alpha
}\right)  ^{r-1}}{\left(  n-1\right)  ^{r-1}}-x^{\alpha\left(  r-1\right)
}\right)  +\delta^{\alpha}x^{\alpha\left(  r-1\right)  },
\end{align*}
and so,%
\[
\left(  \left(  \frac{\lambda n^{r/\alpha-1}}{\left(  r-1\right)  !}\right)
^{\alpha}-\delta^{\alpha}\right)  x^{\alpha\left(  r-1\right)  }\leq
\binom{n-1}{r-1}\delta^{\alpha-1}\left(  \frac{\left(  1-x^{\alpha}\right)
^{r-1}}{\left(  n-1\right)  ^{r-1}}-x^{\alpha\left(  r-1\right)  }\right)  ,
\]
completing the proof of Lemma \ref{le1}.
\end{proof}

\bigskip

If $G$ is an $r$-graph of order $n,$ with minimum degree $\delta,$ from
(\ref{gin}) we see that
\begin{equation}
\frac{\lambda^{\left(  \alpha\right)  }\left(  G\right)  n^{r/\alpha-1}%
}{\left(  r-1\right)  !}\geq\frac{re\left(  G\right)  }{n}\geq\delta.
\label{gin1}%
\end{equation}

In the proof of Lemma \ref{le2}, we shall need the following simple
consequence of this fact.

\begin{proposition}
If $\alpha\geq1,$ $r\geq2,$ and $G$ is an $r$-graph of order $n$, with minimum
degree $\delta$, then
\begin{equation}
\left(  \frac{\lambda^{\left(  \alpha\right)  }\left(  G\right)
n^{r/\alpha-1}}{\left(  r-1\right)  !}\right)  ^{\alpha}-\delta^{\alpha}%
\geq\left(  \frac{\lambda^{\left(  \alpha\right)  }\left(  G\right)
n^{r/\alpha-1}}{\left(  r-1\right)  !}-\delta\right)  \alpha\delta^{\alpha-1}
\label{in13}%
\end{equation}

\end{proposition}

\begin{proof}
Indeed, setting for short
\[
\frac{\lambda^{\left(  \alpha\right)  }\left(  G\right)  n^{r/\alpha-1}%
}{\left(  r-1\right)  !}=b,
\]
inequality (\ref{in13}) is equivalent to
\[
b^{\alpha}-\delta^{\alpha}\geq\left(  b-\delta\right)  \alpha\delta
^{\alpha-1\text{ }}.
\]
But (\ref{gin1}) implies that $b\geq\delta,$ and using Bernoulli's inequality,
we find that
\[
b^{\alpha}-\delta^{\alpha}=\delta^{\alpha}\left(  \frac{b^{a}}{\delta^{\alpha
}}-1\right)  =\delta^{\alpha}\left(  \left(  1+\frac{b-\delta}{\delta}\right)
^{\alpha}-1\right)  \geq\delta^{\alpha}\left(  \frac{\alpha\left(
b-\delta\right)  }{\delta}\right)  =\left(  b-\delta\right)  \alpha
\delta^{\alpha-1},
\]
completing the proof of the proposition.
\end{proof}

\bigskip

Recall that inequality (\ref{gin1}) is obtained by taking a vector $\left(
x_{1},\ldots,x_{n}\right)  $ with $x_{1}=\cdots=x_{n}=n^{-1/\alpha}.$ Now let
$\mathbf{x}=\left(  x_{1},\ldots,x_{n}\right)  $ be a nonnegative vector such
that $\left\Vert \mathbf{x}\right\Vert _{\alpha}=1$ and $\lambda^{\left(
\alpha\right)  }\left(  G\right)  =P_{G}\left(  \mathbf{x}\right)  .$ It turns
out that if the entries of $\mathbf{x}$ are close to $n^{-1/\alpha},$ then the
bound (\ref{gin1}) can be inverted to some extent, which also implies that the
graph is almost regular. The following technical lemma gives a quantitative
form of this statement.

\begin{lemma}
\label{le2} Let $1<\alpha\leq r,$ let $G$ be an $r$-graph of sufficiently
large order $n,$ with minimum degree $\delta$ and with $\lambda^{\left(
\alpha\right)  }\left(  G\right)  =\lambda.$ Let $\mathbf{x}=\left(
x_{1},\ldots,x_{n}\right)  $ be a nonnegative vector such that $\left\Vert
\mathbf{x}\right\Vert _{\alpha}=1$ and $\lambda=P_{G}\left(  \mathbf{x}%
\right)  .$ If the value $x=\min\left\{  x_{1},\ldots,x_{n}\right\}  $
satisfies%
\begin{equation}
x^{\alpha}\geq\frac{1}{n}\left(  1-\frac{1}{\left(  \alpha-1\right)  \log
n}\right)  ,\label{prem}%
\end{equation}
then
\[
\frac{\lambda n^{r/\alpha-1}}{\left(  r-1\right)  !}\leq\delta+\frac
{2r}{\alpha\left(  \alpha-1\right)  \log n}\binom{n-1}{r-1}.
\]

\end{lemma}

\begin{proof}
To begin with, Lemma \ref{le1} gives%
\[
\left(  \left(  \frac{\lambda n^{r/\alpha-1}}{\left(  r-1\right)  !}\right)
^{\alpha}-\delta^{\alpha}\right)  x^{\alpha\left(  r-1\right)  }\leq
\binom{n-1}{r-1}\delta^{\alpha-1}\left(  \frac{\left(  1-x^{\alpha}\right)
^{r-1}}{\left(  n-1\right)  ^{r-1}}-x^{\alpha\left(  r-1\right)  }\right)  .
\]
Now, since the premise (\ref{prem}) implies that $x>0,$ we can rearrange the
above inequality to%
\begin{equation}
\left(  \frac{\lambda n^{r/\alpha-1}}{\left(  r-1\right)  !}\right)  ^{\alpha
}-\delta^{\alpha}\leq\binom{n-1}{r-1}\delta^{\alpha-1}\left(  \frac{1}{\left(
n-1\right)  ^{r-1}}\frac{\left(  1-x^{\alpha}\right)  ^{r-1}}{x^{\alpha\left(
r-1\right)  }}-1\right)  .\label{in8}%
\end{equation}
Obviously, the expression $\left(  1-y\right)  /y$ decreases with $y;$ in
particular, the premise
\[
x^{\alpha}\geq\frac{1}{n}\left(  1-\frac{1}{\left(  \alpha-1\right)  \log
n}\right)
\]
implies that
\begin{align*}
\frac{1-x^{\alpha}}{x^{\alpha}} &  \leq\frac{n-1+\frac{1}{\left(
\alpha-1\right)  \log n}}{1-\frac{1}{\left(  \alpha-1\right)  \log n}}%
=\frac{\left(  n-1\right)  \left(  \alpha-1\right)  \log n+1}{\left(
\alpha-1\right)  \log n-1}\\
&  =\left(  n-1\right)  \frac{\log n+\frac{1}{\left(  n-1\right)  \left(
\alpha-1\right)  }}{\log n-\frac{1}{\left(  \alpha-1\right)  }}\\
&  =\left(  n-1\right)  \left(  1+\frac{n}{\left(  \alpha-1\right)  \left(
n-1\right)  \left(  \log n-\frac{1}{\left(  \alpha-1\right)  }\right)
}\right)  .
\end{align*}
Now, bounding the expression $\left(  1-x^{\alpha}\right)  /x^{\alpha}$ in
(\ref{in8}), we get
\begin{equation}
\left(  \frac{\lambda n^{r/\alpha-1}}{\left(  \left(  r-1\right)  !\right)
}\right)  ^{\alpha}-\delta^{\alpha}<\binom{n-1}{r-1}\delta^{\alpha-1}\left(
\left(  1+\frac{n}{\left(  \alpha-1\right)  \left(  n-1\right)  \left(  \log
n-\frac{1}{\left(  \alpha-1\right)  }\right)  }\right)  ^{r-1}-1\right)
.\label{in9}%
\end{equation}
Next, assuming that $n$ is large enough and using Bernoulli's inequality, we
find that
\begin{align*}
\left(  1+\frac{1}{\left(  \alpha-1\right)  \left(  1-1/n\right)  \left(  \log
n-\frac{1}{\left(  \alpha-1\right)  }\right)  }\right)  ^{r-1} &  \leq
1+\frac{r-1}{\left(  \alpha-1\right)  \left(  1-1/n\right)  \left(  \log
n-\frac{1}{\left(  \alpha-1\right)  }\right)  -r}\\
&  \leq1+\frac{2r}{\left(  \alpha-1\right)  \log n}.
\end{align*}
Replacing this bound in (\ref{in9}), we get
\begin{equation}
\left(  \frac{\lambda n^{r/\alpha-1}}{\left(  \left(  r-1\right)  !\right)
}\right)  ^{\alpha}-\delta^{\alpha}<\binom{n-1}{r-1}\delta^{\alpha-1}\frac
{2r}{\left(  \alpha-1\right)  \log n}.\label{in10}%
\end{equation}
On the other hand, (\ref{in13}) implies
\[
\left(  \frac{\lambda n^{r/\alpha-1}}{\left(  r-1\right)  !}-\delta\right)
\alpha\delta^{\alpha-1}\leq\binom{n-1}{r-1}\delta^{\alpha-1}\frac{2r}{\left(
\alpha-1\right)  \log n},
\]
and so,
\[
\frac{\lambda n^{r/\alpha-1}}{\left(  r-1\right)  !}\leq\delta+\frac
{2r}{\alpha\left(  \alpha-1\right)  \log n}\binom{n-1}{r-1},
\]
completing the proof of Lemma \ref{le2}.
\end{proof}

\medskip

\section{\label{pf}Proofs}

The proofs given below follow the order of appearance of the statements in
Section 1, except for the proof of Theorem \ref{thCon2}, given at the end of
the section.\bigskip

\begin{proof}
[\textbf{Proof of Proposition \ref{pro3}}]Let $\mathbf{x}=\left(  x_{1}%
,\ldots,x_{n}\right)  $ be a nonnegative vector such that $\left\Vert
\mathbf{x}\right\Vert _{\alpha}=1$ and $\lambda^{\left(  \alpha\right)
}\left(  G\right)  =P_{G}\left(  \mathbf{x}\right)  .$ Assume that $\alpha\geq
r,$ let $x=\max\left\{  x_{1},\ldots,x_{n}\right\}  ,$ and let $k\in V\left(
G\right)  $ be a vertex for which $x_{k}=x.$ From equations (\ref{eequ}) we
have%
\[
\frac{\lambda^{\left(  \alpha\right)  }\left(  G\right)  }{\left(  r-1\right)
!}x^{\alpha-1}=\sum_{\left\{  k,i_{1},\ldots,i_{r-1}\right\}  \in E\left(
G\right)  }x_{i_{1}}\cdots x_{i_{r-1}}\leq\Delta x^{r-1}%
\]
Since $x\geq n^{-1/\alpha}$ and $\alpha\geq r,$ we find that%
\[
\frac{\lambda^{\left(  \alpha\right)  }\left(  G\right)  }{\left(  r-1\right)
!}\leq\Delta x^{r-\alpha}\leq\Delta\left(  n^{-1/\alpha}\right)  ^{r-\alpha
}=\frac{\Delta}{n^{r/\alpha-1}},
\]
proving (\ref{inmax}). Now if we have equality in the above, then
$x=n^{-1/\alpha}$ and so $x_{1}=\cdots=x_{n}=n^{-1/\alpha}.$ Thus, equations
(\ref{eequ}) show that all degrees are equal to $\lambda^{\left(
\alpha\right)  }\left(  G\right)  /\left(  r-1\right)  !=\Delta$, and $G$ is
regular. On the other hand,
\[
\frac{re\left(  G\right)  }{n^{r/\alpha}}\leq\frac{\lambda^{\left(
\alpha\right)  }\left(  G\right)  }{\left(  r-1\right)  !}\leq\frac{\Delta
}{n^{r/\alpha-1}},
\]
and so if $G$ is regular, then $\lambda^{\left(  \alpha\right)  }\left(
G\right)  /\left(  r-1\right)  !=\Delta n^{1-r/\alpha},$ completing the proof
of \emph{(i)}.

To prove \emph{(ii)}, assume that $1\leq\alpha<r$ and let $G$ be the union of
$n$ disjoint complete $r$-graph on $k>r$ vertices. It is easy to see that
\[
\lambda^{\left(  \alpha\right)  }\left(  G_{0}\right)  =P_{G_{0}}\left(
\left(  k^{-1/\alpha},\ldots,k^{-1/\alpha}\right)  \right)  =r!\binom{k}%
{r}k^{-r/\alpha}%
\]
and $\Delta=\binom{k-1}{r-1}.$ Hence,%
\[
\frac{\lambda^{\left(  \alpha\right)  }\left(  G_{0}\right)  }{\left(
r-1\right)  !}=\frac{\Delta}{k^{r/\alpha-1}},
\]
and so (\ref{inmax}) fails if $n$ is large, because $r>\alpha.$ This completes
the proof of \emph{(ii)}.

There is nothing to prove in \emph{(iii)}, in view of $\lambda^{\left(
\alpha\right)  }\left(  G\right)  <\lambda^{\left(  r\right)  }\left(
G\right)  \leq\left(  r-1\right)  !\Delta.$
\end{proof}

\bigskip

\begin{proof}
[\textbf{Proof of Proposition \ref{pro4}}]If $\lambda^{\left(  \alpha\right)
}\left(  G\right)  =r!e\left(  G\right)  /n^{r/\alpha},$ then for the
$n$-vector $\mathbf{x}=\left(  n^{-1/\alpha},\ldots,n^{-1/\alpha}\right)  $ we
see that $\lambda^{\left(  \alpha\right)  }\left(  G\right)  =P_{G}\left(
\mathbf{x}\right)  ,$ and so $\mathbf{x}$ satisfies equations (\ref{eequ}),
which implies that all degrees are equal. If $\alpha\geq r$ and $G$ is
regular, then from%
\[
\frac{re\left(  G\right)  }{n^{r/\alpha}}\leq\frac{\lambda^{\left(
\alpha\right)  }\left(  G\right)  }{\left(  r-1\right)  !}\leq\frac{\Delta
}{n^{r/\alpha-1}},\text{ }%
\]
we see that $\lambda^{\left(  \alpha\right)  }\left(  G\right)  =r!e\left(
G\right)  /n^{r/\alpha}.$ Finally, let $1\leq\alpha<r,$ fix an integer $k>r,$
and take the union of $n$ complete $r$-graphs of order $k.$ As in the proof of
Proposition \ref{pro3} we see that
\[
\lambda^{\left(  \alpha\right)  }\left(  G\right)  =r!\binom{k}{r}%
k^{-r/\alpha}>r!n\binom{k}{r}/n^{r/\alpha}=r!e\left(  G\right)  /n^{r/\alpha
},
\]
completing the proof of Proposition \ref{pro4}.
\end{proof}

\bigskip

\begin{proof}
[\textbf{Proof of Proposition \ref{pro5}}]Let $1\leq\alpha<\beta,$ and let
$\mathbf{x}=\left(  x_{1},\ldots,x_{n}\right)  $ be a nonnegative vector
satisfying $\left\Vert \mathbf{x}\right\Vert _{\beta}=1$ and $\lambda^{\left(
\beta\right)  }\left(  G\right)  =P_{G}\left(  \mathbf{x}\right)  $. Using
Jensen's inequality, we see that
\begin{equation}
\left\Vert \mathbf{x}\right\Vert _{\alpha}^{r}=\left(  x_{1}^{\alpha}%
+\cdots+x_{n}^{\alpha}\right)  ^{r/\alpha}\leq\left(  x_{1}^{\beta}%
+\cdots+x_{n}^{\beta}\right)  ^{r/\beta}n^{r/\alpha-r/\beta}=n^{r/\alpha
-r/\beta}. \label{in12}%
\end{equation}
Therefore,%
\[
\lambda^{\left(  \beta\right)  }\left(  G\right)  =P_{G}\left(  \mathbf{x}%
\right)  \leq\lambda^{\left(  \alpha\right)  }\left(  G\right)  \left\Vert
\mathbf{x}\right\Vert _{\alpha}^{r}\leq\lambda^{\left(  \alpha\right)
}\left(  G\right)  n^{r/\alpha-r/\beta},
\]
implying that $h_{G}\left(  \alpha\right)  $ is nonincreasing.

If $G$ is non-regular, then some of the entries $x_{1},\ldots,x_{n}$ are
distinct and so Jensen's inequality implies strict inequality in (\ref{in12}),
implying in turn that
\[
\lambda^{\left(  \beta\right)  }\left(  G\right)  <\lambda^{\left(
\alpha\right)  }\left(  G\right)  n^{r/\alpha-r/\beta},
\]
proving Proposition \ref{pro5}.
\end{proof}

\bigskip

\begin{proof}
[\textbf{Proof of Propostion \ref{pro6}}]Let $1\leq\alpha<\beta.$ Set for
short $m=e\left(  G\right)  $ and let $\mathbf{x}=\left(  x_{1},\ldots
,x_{n}\right)  $ be a nonnegative vector satisfying $\left\Vert \mathbf{x}%
\right\Vert _{\beta}=1$ and $\lambda^{\left(  \beta\right)  }\left(  G\right)
=P_{G}\left(  \mathbf{x}\right)  .$ Using Jensen's inequality, we see that
\[
\frac{\lambda^{\left(  \beta\right)  }\left(  G\right)  }{r!m}=\frac{1}{m}%
\sum_{\left\{  i_{1},\ldots,i_{r}\right\}  \in E\left(  G\right)  }x_{i_{1}%
}\ldots x_{i_{r}}\leq\left(  \frac{1}{m}\sum_{\left\{  i_{1},\ldots
,i_{r}\right\}  \in E\left(  G\right)  }x_{i_{1}}^{\beta/\alpha}\ldots
x_{i_{r}}^{\beta/\alpha}\right)  ^{\alpha/\beta}.
\]
Note that
\[
\left(  x_{1}^{\beta/\alpha}\right)  ^{\alpha}+\cdots+\left(  x_{n}%
^{\beta/\alpha}\right)  ^{\alpha}=x_{1}^{\beta}+\cdots+x_{n}^{\beta}=1,
\]
and so, for the vector $\mathbf{y}=\left(  x_{1}^{\beta/\alpha},\ldots
,x_{n}^{\beta/\alpha}\right)  $ we have $\left\Vert \mathbf{y}\right\Vert
_{\alpha}=1.$ Hence,
\[
\frac{1}{m}\sum_{\left\{  i_{1},\ldots,i_{r}\right\}  \in E\left(  G\right)
}x_{i_{1}}^{\beta/\alpha}\ldots x_{i_{r}}^{\beta/\alpha}=\frac{1}{r!m}%
P_{G}\left(  \mathbf{y}\right)  \leq\frac{1}{r!m}\lambda^{\left(
\alpha\right)  }\left(  G\right)  ,
\]
and so,%
\[
\left(  \frac{\lambda^{\left(  \beta\right)  }\left(  G\right)  }{r!m}\right)
^{\beta}\leq\left(  \frac{\lambda^{\left(  \alpha\right)  }\left(  G\right)
}{r!m}\right)  ^{\alpha},
\]
proving Proposition \ref{pro6}.
\end{proof}

\bigskip

\begin{proof}
[\textbf{Proof of Theorem \ref{th1}}]For every integer $n\geq2,$ set for short
$\lambda_{n}^{\left(  \alpha\right)  }=\lambda^{\left(  \alpha\right)
}\left(  \mathcal{P},n\right)  .$ Let $G\in\mathcal{P}_{n}$ be such that
$\lambda^{\left(  \alpha\right)  }\left(  G\right)  =\lambda_{n}^{\left(
\alpha\right)  }$ and let $\mathbf{x}=\left(  x_{1},\ldots,x_{n}\right)  $ be
a nonnegative vector such that $\left\Vert \mathbf{x}\right\Vert _{\alpha}=1$
and
\[
\lambda_{n}^{\left(  \alpha\right)  }=\lambda^{\left(  \alpha\right)  }\left(
G\right)  =P_{G}\left(  \mathbf{x}\right)  .
\]
If $\alpha=1,$ we obviously have $\lambda_{n}^{\left(  1\right)  }\geq
\lambda_{n-1}^{\left(  1\right)  }.$ and in view of
\[
\lambda_{n}^{\left(  1\right)  }=P_{G}\left(  \mathbf{x}\right)  \leq
r!\sum_{1\leq i_{1}<\cdots<i_{r}\leq n}x_{i_{1}}\ldots x_{i_{r}}\leq\left(
x_{1}+\cdots+x_{n}\right)  ^{r}=1,
\]
we see that the sequence $\left\{  \lambda_{n}^{\left(  1\right)  }\right\}
_{n=1}^{\infty}$ is converging to some $\lambda$. Then,%
\[
\lambda=\lim_{n\rightarrow\infty}\lambda_{n}^{\left(  1\right)  }%
n^{r-r}=\lambda^{\left(  1\right)  }\left(  \mathcal{P}\right)  ,
\]
proving (\ref{bnd1})

Suppose now that $\alpha>1.$ Obviously there exists a vertex $k$ of $G$ such
that $x_{k}^{\alpha}\leq1/n.$ Write $G-k$ for the $r$-graph obtained from
$\mathbf{x}$ by omitting the vertex $k,$ and let $\mathbf{x}^{\prime}$ be the
$\left(  n-1\right)  $-vector obtained from $\mathbf{x}$ by omitting the entry
$x_{k}.$ Then,
\begin{align*}
P_{G-k}\left(  \mathbf{x}^{\prime}\right)   &  =P_{G}\left(  \mathbf{x}%
\right)  -r!x_{k}\sum_{\left\{  k,i_{1},\ldots,i_{r-1}\right\}  \in E\left(
G\right)  }x_{i_{1}}\ldots x_{i_{r-1}}=\\
&  =\lambda^{\left(  \alpha\right)  }\left(  G\right)  -rx_{i}\left(
\lambda^{\left(  \alpha\right)  }\left(  G\right)  x_{i}^{\alpha-1}\right)
=\lambda_{n}^{\left(  \alpha\right)  }\left(  1-rx_{i}^{\alpha}\right)
\end{align*}
On the other hand, $\mathcal{P}$ is a hereditary property, so $G-k\in
\mathcal{P}_{n-1},$ and therefore,
\[
P_{G-k}\left(  \mathbf{x}^{\prime}\right)  \leq\lambda^{\left(  \alpha\right)
}\left(  G-k\right)  \left\Vert \mathbf{x}^{\prime}\right\Vert _{\alpha}%
^{r}=\lambda^{\left(  \alpha\right)  }\left(  G-k\right)  \left(
1-x_{k}^{\alpha}\right)  ^{r/\alpha}\leq\lambda_{n-1}^{\left(  \alpha\right)
}\left(  1-x_{k}^{\alpha}\right)  ^{r/\alpha}.
\]
Thus, we obtain%
\begin{equation}
\lambda_{n}^{\left(  \alpha\right)  }\leq\lambda_{n-1}^{\left(  \alpha\right)
}\frac{\left(  1-x_{k}^{\alpha}\right)  ^{r/\alpha}}{\left(  1-rx_{k}^{\alpha
}\right)  }.\label{in3}%
\end{equation}
Note that the function
\[
f\left(  x\right)  =\frac{\left(  1-x\right)  ^{r/\alpha}}{1-rx}%
\]
is nondecreasing in $x$ for $0\leq x\leq1/n$ and $n$ sufficiently large.
Indeed,%
\begin{align*}
\frac{df\left(  x\right)  }{dx} &  =\frac{-\frac{r}{\alpha}\left(  1-x\right)
^{r/\alpha-1}\left(  1-rx\right)  +rx\left(  1-x\right)  ^{r/\alpha}}{\left(
1-rx\right)  ^{2}}\\
&  =\left(  -\frac{1}{\alpha}\left(  1-rx\right)  +\left(  1-x\right)
\right)  \frac{r\left(  1-x\right)  ^{r/\alpha-1}}{\left(  1-rx\right)  ^{2}%
}\\
&  =\left(  -\left(  \frac{1}{\alpha}-1\right)  +\left(  \frac{r}{\alpha
}-1\right)  x\right)  \frac{r\left(  1-x\right)  ^{r/\alpha-1}}{\left(
1-rx\right)  ^{2}}\geq0
\end{align*}
Here we use the fact that $1/\alpha-1>0$ and that $\left(  r/\alpha-1\right)
x$ tends to $0$ when $n$ $\rightarrow\infty.$

Hence, in view of (\ref{in3}), we find that for $n$ large enough,
\[
\lambda_{n}^{\left(  \alpha\right)  }\leq\lambda_{n-1}^{\left(  \alpha\right)
}f\left(  x_{k}^{\alpha}\right)  \leq\lambda_{n-1}^{\left(  \alpha\right)
}f\left(  \frac{1}{n}\right)  =\lambda_{n-1}^{\left(  \alpha\right)  }%
\frac{n\left(  1-1/n\right)  ^{r/\alpha}}{\left(  n-r\right)  },
\]
and so, (\ref{in3}) implies that
\[
\frac{\lambda_{n}^{\left(  \alpha\right)  }n^{r/\alpha-1}}{\left(  n-1\right)
\left(  n-2\right)  \cdots\left(  n-r+1\right)  }\leq\frac{\lambda
_{n}^{\left(  \alpha\right)  }\left(  n-1\right)  ^{r/\alpha-1}}{\left(
n-2\right)  \left(  n-2\right)  \cdots\left(  n-r\right)  }.
\]
Therefore, the sequence%
\[
\left\{  \frac{\lambda_{n}^{\left(  \alpha\right)  }n^{r/\alpha-1}}{\left(
n-1\right)  \left(  n-2\right)  \cdots\left(  n-r+1\right)  }\right\}
_{n=1}^{\infty}%
\]
is nonincreasing, and so it is converging, completing the proof of
(\ref{exlima}) and (\ref{bndsa}) for $\alpha>1$.
\end{proof}

\bigskip

\begin{proof}
[\textbf{Proof of Proposition \ref{pro7}}]For every $G\in\mathcal{P}$,
Proposition \ref{th2} gives
\[
\left(  \lambda^{\left(  \beta\right)  }\left(  G\right)  \right)  ^{\beta
}\leq\left(  \lambda^{\left(  \alpha\right)  }\left(  G\right)  \right)
^{\alpha}\left(  r!e\left(  G\right)  \right)  ^{\beta-\alpha}.
\]
Hence, choosing $G\in\mathcal{P}_{n}$ such that $\lambda^{\left(
\beta\right)  }\left(  G\right)  =\lambda^{\left(  \beta\right)  }\left(
\mathcal{P},n\right)  ,$ we find that
\begin{align*}
\left(  \frac{\lambda^{\left(  \beta\right)  }\left(  \mathcal{P},n\right)
n^{r/\beta-1}}{\left(  n-1\right)  _{r-1}}\right)  ^{\beta} &  =\left(
\frac{\lambda^{\left(  \beta\right)  }\left(  G\right)  n^{r/\beta-1}}{\left(
n-1\right)  _{r-1}}\right)  ^{\beta}\\
&  \leq\left(  \frac{\lambda^{\left(  \beta\right)  }\left(  G\right)
n^{r/\beta-1}}{\left(  n-1\right)  _{r-1}}\right)  ^{\alpha}\left(
\frac{r!e\left(  G\right)  }{\left(  n\right)  _{r}}\right)  ^{\beta-\alpha}\\
&  \leq\left(  \frac{\lambda^{\left(  \beta\right)  }\left(  G\right)
n^{r/\beta-1}}{\left(  n-1\right)  _{r-1}}\right)  ^{\alpha}\left(  ex\left(
\mathcal{P},n\right)  /\binom{n}{r}\right)  ^{\beta-\alpha}.
\end{align*}
Letting now $n\rightarrow\infty,$ we see that%
\[
\left(  \lambda^{\left(  \beta\right)  }\left(  \mathcal{P}\right)  \right)
^{\beta}\leq\left(  \lambda^{\left(  \alpha\right)  }\left(  \mathcal{P}%
\right)  \right)  ^{\alpha}\left(  \pi\left(  \mathcal{P}\right)  \right)
^{\beta-\alpha}%
\]
completing the proof of (\ref{inabp}).
\end{proof}

\bigskip

\begin{proof}
[\textbf{Proof of Theorem \ref{thblo}}]For the purposes of this proof let us
write $k_{H}\left(  G\right)  $ for the number of subgraphs of $G$ which are
isomorphic to $H$.

We start by recalling the Hypergraph Removal Lemma, one of the most important
consequences of the Hypergraph Regularity Lemma, proved independently by
Gowers \cite{Gow07} and by Nagle, R\"{o}dl, Schacht and Skokan \cite{NRS06},
\cite{RoSk04}.\medskip

\textbf{Removal Lemma} \emph{Let }$H$\emph{ be an }$r$\emph{-graph of order
}$h$\emph{ and let }$\varepsilon>0.$\emph{ There exists }$\delta=\delta
_{H}\left(  \varepsilon\right)  >0$\emph{ such that if }$G$\emph{ is an }%
$r$\emph{-graph of order }$n,$\emph{ with }$k_{H}\left(  G\right)  <\delta
n^{h},$\emph{ then there is an }$r$\emph{-graph }$G_{0}\subset G$\emph{ such
that }$e\left(  G_{0}\right)  \geq e\left(  G\right)  -\varepsilon n^{r}%
$\emph{ and }$k_{H}\left(  G_{0}\right)  =0.\medskip$

In \cite{Erd64} Erd\H{o}s showed that for every $\varepsilon>0$ there exists
$\delta>0$ such that if $G$ is an $r$-graph with $e\left(  G\right)
\geq\varepsilon n^{r},$ then $K_{r}\left(  k,\ldots,k\right)  \subset G$ for
some $k\geq\delta\left(  \log n\right)  ^{1/\left(  r-1\right)  }.$ As noted
by R\"{o}dl and Schacht \cite{RoSc12} (also by Bollob\'{a}s, unpublished) this
result of Erd\H{o}s implies the following general assertion.\medskip

\textbf{Theorem A }\emph{Let }$H$\emph{ be an }$r$\emph{-graph of order }%
$h$\emph{ and let }$\varepsilon>0.$\emph{ There exists }$\delta=\delta
_{H}\left(  \varepsilon\right)  >0$\emph{ such that if }$G$\emph{ is an }%
$r$\emph{-graph of order }$n,$\emph{ with }$k_{H}\left(  G\right)
\geq\varepsilon n^{h},$\emph{ then }$H\left(  k,\ldots,k\right)  \subset
G$\emph{ for some }$k=\left\lceil \delta\left(  \log n\right)  ^{1/\left(
h-1\right)  }\right\rceil .\medskip$

Suppose now that $H$ is an $r$-graph of order $h,$ let $H\left(  k_{1}%
,\ldots,k_{h}\right)  $ be a fixed blow-up of $H,$ and set $k=\max\left\{
k_{1},\ldots,k_{h}\right\}  .$ Take $G\in Mon\left(  H\left(  k_{1}%
,\ldots,k_{h}\right)  \right)  _{n}$ such that
\[
\lambda^{\left(  \alpha\right)  }\left(  G\right)  =\lambda^{\left(
\alpha\right)  }\left(  Mon\left(  H\left(  k_{1},\ldots,k_{h}\right)
\right)  ,n\right)
\]
For every $\varepsilon>0,$ choose $\delta=\delta_{H}\left(  \varepsilon
\right)  ,$ as in the Removal Lemma. Since $H\left(  k,\ldots,k\right)
\nsubseteq G,$ Theorem A implies that if $n$ is sufficiently large, then
$k_{H}\left(  G\right)  <\delta n^{h}.$ Now the Removal Lemma implies that
there is an $r$-graph $G_{0}\subset G$ such that $e\left(  G_{0}\right)  \geq
e\left(  G\right)  -\varepsilon n^{r}$ and $k_{H}\left(  G_{0}\right)  =0.$
Clearly, we can assume that $V\left(  G_{0}\right)  =V\left(  G\right)  .$ By
Proposition \ref{pro10}, we see that
\[
\lambda^{\left(  \alpha\right)  }\left(  G\right)  \leq\lambda^{\left(
\alpha\right)  }\left(  G_{0}\right)  +\left(  \varepsilon r!n^{r}\right)
^{1-1/\alpha},
\]
and hence,
\begin{align*}
\frac{\lambda^{\left(  \alpha\right)  }\left(  Mon\left(  H\left(
k,\ldots,k\right)  \right)  ,n\right)  n^{r/\alpha-1}}{\left(  n-1\right)
_{r-1}} &  \leq\frac{\left(  \lambda^{\left(  \alpha\right)  }\left(
G_{0}\right)  +\left(  \varepsilon r!n^{r}\right)  ^{1-1/\alpha}\right)
n^{r/\alpha-1}}{\left(  n-1\right)  _{r-1}}\\
&  \leq\lambda^{\left(  \alpha\right)  }\left(  Mon\left(  H\right)  \right)
+o\left(  1\right)  +\frac{\left(  \varepsilon r!\right)  ^{1-1/\alpha
}n^{r-r/\alpha}n^{1-r/\alpha}}{n^{r-1}}\\
&  =\lambda^{\left(  \alpha\right)  }\left(  Mon\left(  H\right)  \right)
+o\left(  1\right)  +\left(  \varepsilon r!\right)  ^{1-1/\alpha}.
\end{align*}
Since $\varepsilon$ can be made arbitrarily small, we see that
\[
\lambda^{\left(  \alpha\right)  }\left(  MonH\left(  k_{1},\ldots
,k_{h}\right)  \right)  \leq\lambda^{\left(  \alpha\right)  }\left(
Mon\left(  H\right)  \right)  ,
\]
completing the proof of Theorem \ref{thblo}.
\end{proof}

\bigskip

\begin{proof}
[\textbf{Proof Theorem \ref{th2}}]Let $1\leq\alpha<\beta.$ Using
(\ref{inabp}), after some cancellations, we find that
\[
\lambda^{\left(  \beta\right)  }\left(  \mathcal{P}\right)  \leq\left(
\lambda^{\left(  \alpha\right)  }\left(  \mathcal{P}\right)  \right)
^{\alpha/\beta}\left(  \pi\left(  \mathcal{P}\right)  \right)  ^{1-\alpha
/\beta}.
\]
From (\ref{limgin}) we have $\pi\left(  \mathcal{P}\right)  \leq
\lambda^{\left(  \alpha\right)  }\left(  \mathcal{P}\right)  $ and so
\[
\left(  \pi\left(  \mathcal{P}\right)  \right)  ^{1-\alpha/\beta}\leq\left(
\lambda^{\left(  \alpha\right)  }\left(  \mathcal{P}\right)  \right)
^{1-\alpha/\beta}.
\]
Substituting in the above, we see that
\[
\lambda^{\left(  \beta\right)  }\left(  \mathcal{P}\right)  \leq\left(
\lambda^{\left(  \alpha\right)  }\left(  \mathcal{P}\right)  \right)
^{\alpha/\beta}\left(  \lambda^{\left(  \alpha\right)  }\left(  \mathcal{P}%
\right)  \right)  ^{1-\alpha/\beta}=\lambda^{\left(  \alpha\right)  }\left(
\mathcal{P}\right)  ,
\]
proving Theorem \ref{th2}.
\end{proof}

\bigskip

\begin{proof}
[\textbf{Proof of Theorem \ref{th3}}]As mentioned above, if $\lambda\left(
\mathcal{P}\right)  =\pi\left(  \mathcal{P}\right)  ,$ then
\[
\lambda^{\left(  \alpha\right)  }\left(  \mathcal{P}\right)  =\pi\left(
\mathcal{P}\right)
\]
for every $\alpha>r.$ Thus, all we we need to prove is the case $1<\alpha\leq
r$. Fix $\alpha$ in the indicated range, and for every natural $n,$ set for
short $\lambda_{n}=\lambda^{\left(  \alpha\right)  }\left(  \mathcal{P}%
,n\right)  .$ Note that if $\lambda^{\left(  \alpha\right)  }\left(
\mathcal{P}\right)  =0,$ then inequality (\ref{limgin}) implies that
$\pi\left(  \mathcal{P}\right)  =0;$ therefore Theorem \ref{th3} holds for
$\lambda^{\left(  \alpha\right)  }\left(  \mathcal{P}\right)  =0.$ We shall
assume hereafter that $\lambda^{\left(  \alpha\right)  }\left(  \mathcal{P}%
\right)  >0.$ Recall also that $\left(  n\right)  _{r}$ stands for $n!/\left(
n-k\right)  !$.\medskip

\textbf{Claim A} \emph{There are infinitely many }$n$\emph{ for which}
\[
\frac{\lambda_{n-1}\left(  n-1\right)  ^{r/\alpha-1}}{\left(  n-2\right)
_{r-1}}-\frac{\lambda_{n}n^{r/\alpha-1}}{\left(  n-1\right)  _{r-1}}<\frac
{1}{n\log n}\cdot\frac{\lambda_{n}n^{r/\alpha-1}}{\left(  n-1\right)  _{r-1}}%
\]
\medskip

\textbf{Proof }Indeed, assume for a contradiction that there is $n_{0}$ such
that
\[
\frac{\lambda_{n-1}\left(  n-1\right)  ^{r/\alpha-1}}{\left(  n-2\right)
_{r-1}}-\frac{\lambda_{n}n^{r/\alpha-1}}{\left(  n-1\right)  _{r-1}}\geq
\frac{1}{n\log n}\cdot\frac{\lambda_{n}n^{r/\alpha-1}}{\left(  n-1\right)
_{r-1}}%
\]
for every $n>n_{0}.$ Then for every $k>n_{0}$ we see that
\begin{align*}
\frac{\lambda_{n_{0}-1}\left(  n_{0}-1\right)  ^{r/\alpha-1}}{\left(
n_{0}-1\right)  _{r-1}}-\frac{\lambda_{k}k^{r/\alpha-1}}{\left(  k-1\right)
_{r-1}} &  =\sum_{n=n_{0}}^{k}\frac{\lambda_{n-1}\left(  n-1\right)
^{r/\alpha-1}}{\left(  n-2\right)  _{r-1}}-\frac{\lambda_{n}n^{r/\alpha-1}%
}{\left(  n-1\right)  _{r-1}}\\
&  \geq\sum_{n=n_{0}}^{k}\frac{1}{n\log n}\frac{\lambda_{n}n^{r/\alpha-1}%
}{\left(  n-1\right)  _{r-1}}\\
&  \geq\lambda^{\left(  \alpha\right)  }\left(  \mathcal{P}\right)
\sum_{n=n_{0}}^{k}\frac{1}{n\log n}.
\end{align*}
This is a contradiction, since the left-hand side is bounded and the
right-hand side diverges, proving Claim A.

Using Claim A, choose sufficiently large $n$ so that
\[
\frac{\lambda_{n-1}\left(  n-1\right)  ^{r/\alpha-1}}{\left(  n-2\right)
_{r-1}}-\frac{\lambda_{n}n^{r/\alpha-1}}{\left(  n-1\right)  _{r-1}}<\frac
{1}{n\log n}\cdot\frac{\lambda_{n}n^{r/\alpha-1}}{\left(  n-1\right)  _{r-1}}.
\]
After some rearrangement we obtain
\begin{equation}
\frac{\lambda_{n-1}}{\lambda_{n}}<\frac{n^{r/\alpha-1}\left(  n-r\right)
}{\left(  n-1\right)  ^{r/\alpha}}\left(  1+\frac{1}{n\log n}\right)
.\label{in11}%
\end{equation}
Let $G\in\mathcal{P}_{n}$ be such that $\lambda^{\left(  \alpha\right)
}\left(  G\right)  =\lambda_{n}$ and let $\mathbf{x}=\left(  x_{1}%
,\ldots,x_{n}\right)  $ be a nonnegative vector satisfying $\left\Vert
\mathbf{x}\right\Vert _{\alpha}=1$ and%
\[
\lambda_{n}=\lambda^{\left(  \alpha\right)  }\left(  G\right)  =P_{G}\left(
\mathbf{x}\right)  .
\]
Let $k\in V\left(  G\right)  $ be a vertex and let $\mathbf{x}^{\prime}$ be
the $\left(  n-1\right)  $-vector obtained from $\mathbf{x}$ by omitting
$x_{k}.$ For the graph $G-k$ we have
\begin{align}
P_{G-k}\left(  \mathbf{x}^{\prime}\right)   &  =P_{G}\left(  \mathbf{x}%
\right)  -r!x_{k}\sum_{\left\{  k,i_{1},\ldots,i_{r-1}\right\}  \in E\left(
G\right)  }x_{i_{1}}\ldots x_{i_{r-1}}\nonumber\\
&  =\lambda^{\left(  \alpha\right)  }\left(  G\right)  -rx_{i}\left(
\lambda^{\left(  \alpha\right)  }\left(  G\right)  x_{k}^{\alpha-1}\right)
=\lambda_{n}\left(  1-rx_{k}^{\alpha}\right)  \label{eq2}%
\end{align}
On the other hand, $\mathcal{P}$ is a hereditary property and so
$G-k\in\mathcal{P}_{n-1}.$ Therefore,
\[
P_{G-i}\left(  \mathbf{x}^{\prime}\right)  \leq\lambda\left(  G-k\right)
\left\Vert \mathbf{x}^{\prime}\right\Vert _{\alpha}^{r}=\lambda\left(
G-k\right)  \left(  1-x_{k}^{\alpha}\right)  ^{r/\alpha}\leq\lambda
_{n-1}\left(  1-x_{k}^{\alpha}\right)  ^{r/\alpha}.
\]
This inequality, together with (\ref{eq2}), implies that%
\[
\lambda_{n}\left(  1-rx_{k}^{\alpha}\right)  \leq\lambda_{n-1}\left(
1-x_{k}^{\alpha}\right)  ^{r/\alpha}.
\]
Hence, in view of (\ref{in11}),%
\begin{equation}
\frac{1-rx_{k}^{\alpha}}{\left(  1-x_{k}^{\alpha}\right)  ^{r/\alpha}}%
\leq\frac{\lambda_{n-1}}{\lambda_{n}}\leq\frac{n^{r/\alpha-1}\left(
n-r\right)  }{\left(  n-1\right)  ^{r/\alpha}}\left(  1+\frac{1}{n\log
n}\right)  .\label{in14}%
\end{equation}
We shall prove that $x_{k}^{\alpha}$ is sufficiently large to apply Lemma
\ref{le2}.\medskip

\textbf{Claim B }\emph{For }$n$\emph{ sufficiently large,}
\[
x_{k}^{\alpha}>\frac{1}{n}\left(  1-\frac{1}{\left(  \alpha-1\right)  \log
n}\right)  .
\]
\medskip

\textbf{Proof }Indeed, assume for a contradiction that
\begin{equation}
x_{k}^{\alpha}\leq\frac{1}{n}\left(  1-\frac{1}{\left(  \alpha-1\right)  \log
n}\right)  .\label{asin}%
\end{equation}
Note that the function
\[
f\left(  y\right)  =\frac{1-ry}{\left(  1-y\right)  ^{r/\alpha}}%
\]
is decreasing in $y$ for $0\leq y<1,$ because
\begin{align*}
\frac{df\left(  y\right)  }{dy} &  =\frac{-r\left(  1-y\right)  ^{r/\alpha
}+\frac{r}{\alpha}\left(  1-ry\right)  \left(  1-y\right)  ^{r/\alpha-1}%
}{\left(  1-y\right)  ^{2r/\alpha}}\\
&  =\frac{r\left(  1-y\right)  ^{r/\alpha-1}}{\alpha\left(  1-y\right)
^{2r/\alpha}}\left(  -\alpha\left(  1-y\right)  +1-ry\right)  \\
&  =\frac{r\left(  1-y\right)  ^{r/\alpha-1}}{\alpha\left(  1-y\right)
^{2r/\alpha}}\left(  -\left(  \alpha-1\right)  -\left(  r-1\right)  y\right)
\leq0.
\end{align*}
Hence, (\ref{asin}) implies that
\[
\frac{1-\frac{r}{n}\left(  1-\frac{1}{\left(  \alpha-1\right)  \log n}\right)
}{1-\frac{1}{n}\left(  1-\frac{1}{\left(  \alpha-1\right)  \log n}\right)
^{r/\alpha}}=f\left(  \frac{1}{n}\left(  1-\frac{1}{\left(  \alpha-1\right)
\log n}\right)  \right)  \leq f\left(  x_{k}^{\alpha}\right)  =\frac
{1-rx_{k}^{\alpha}}{\left(  1-x_{k}^{\alpha}\right)  ^{r/\alpha}}.
\]
Combining this inequality with (\ref{in14}), we see that
\[
\frac{1-\frac{r}{n}\left(  1-\frac{1}{\left(  \alpha-1\right)  \log n}\right)
}{1-\frac{1}{n}\left(  1-\frac{1}{\left(  \alpha-1\right)  \log n}\right)
^{r/\alpha}}\leq\frac{n^{r/\alpha-1}\left(  n-r\right)  }{\left(  n-1\right)
^{r/\alpha}}\left(  1+\frac{1}{n\log n}\right)  ,
\]
and so,%
\[
\frac{\left(  n-r+\frac{r}{\left(  \alpha-1\right)  \log n}\right)
n^{r/\alpha-1}}{\left(  n-1+\frac{1}{\left(  \alpha-1\right)  \log n}\right)
^{r/\alpha}}=\frac{1-\frac{r}{n}\left(  1-\frac{1}{\left(  \alpha-1\right)
\log n}\right)  }{\left(  1-\frac{1}{n}\left(  1-\frac{1}{\left(
\alpha-1\right)  \log n}\right)  \right)  ^{r/\alpha}}\leq\frac{n^{r/\alpha
-1}\left(  n-r\right)  }{\left(  n-1\right)  ^{r/\alpha}}\left(  1+\frac
{1}{n\log n}\right)  .
\]
Rearranging this inequality, we obtain%
\begin{equation}
1+\frac{r}{\left(  \alpha-1\right)  \left(  n-r\right)  \log n}\leq\left(
1+\frac{1}{\left(  \alpha-1\right)  \left(  n-1\right)  \log n}\right)
^{r/\alpha}\left(  1+\frac{1}{n\log n}\right)  \label{in15}%
\end{equation}
To simplify the right-hand side of (\ref{in15}), using Bernoulli's inequality,
we obtain for $n$ sufficiently large,%
\begin{align*}
\left(  1+\frac{1}{\left(  \alpha-1\right)  \left(  n-1\right)  \log
n}\right)  ^{r/\alpha} &  =\frac{1}{\left(  1-\frac{1}{\left(  \alpha
-1\right)  \left(  n-1\right)  \log n+1}\right)  ^{r/\alpha}}\leq\frac
{1}{1-\frac{r/\alpha}{\left(  \alpha-1\right)  \left(  n-1\right)  \log n+1}%
}\\
&  =1+\frac{r/\alpha}{\left(  \alpha-1\right)  \left(  n-1\right)  \log
n-r/\alpha+1}\\
&  \leq1+\frac{r/\alpha}{\left(  \alpha-1\right)  \left(  n-2\right)  \log n}.
\end{align*}
Hence,%
\begin{align*}
1+\frac{r}{\left(  \alpha-1\right)  \left(  n-r\right)  \log n} &  \leq\left(
1+\frac{r/\alpha}{\left(  n-2\right)  \left(  \alpha-1\right)  \log n}\right)
\left(  1+\frac{1}{n\log n}\right)  \\
&  \leq1+\frac{r/\alpha}{\left(  n-2\right)  \left(  \alpha-1\right)  \log
}+\frac{1}{n\log n}+\frac{r/\alpha}{\left(  \alpha-1\right)  n\left(
n-2\right)  \log^{2}n}.
\end{align*}
After some cancellations and rearranging, we obtain%
\[
\frac{r\alpha}{n-r}\leq\frac{r}{n-2}+\frac{\left(  \alpha-1\right)  \alpha}%
{n}+\frac{r}{n\left(  n-2\right)  \log n}.
\]
Now in view of $r\geq2$ and $\alpha\leq r,$ we see that
\begin{align*}
\frac{r\alpha}{n-2} &  \leq\frac{r\alpha}{n-r}\leq\frac{r}{n-2}+\frac{\left(
\alpha-1\right)  \alpha}{n}+\frac{r}{n\left(  n-2\right)  \log n}\\
&  \leq\frac{r}{n-2}+\frac{\left(  \alpha-1\right)  r}{n}+\frac{r}{n\left(
n-2\right)  \log n},
\end{align*}
finally reducing to
\[
2\left(  \alpha-1\right)  \leq\frac{1}{\log n},
\]
which is a contradiction for $n$ large. This completes the proof of Claim B.

Hence, if $n$ is sufficiently large, for every $k\in V\left(  G\right)  ,$ we
have
\[
x_{k}^{r}>\frac{1}{n}\left(  1-\frac{1}{\left(  \alpha-1\right)  \log
n}\right)  ,
\]
and so, Lemma \ref{le2} implies that
\[
\frac{\lambda_{n}n^{r/\alpha-1}}{\left(  r-1\right)  !}\leq\delta+\frac
{2r}{\alpha\left(  \alpha-1\right)  \log n}\binom{n-1}{r-1}\leq\frac{re\left(
G\right)  }{n}+\frac{2r}{\alpha\left(  \alpha-1\right)  \log n}\binom
{n-1}{r-1}.
\]
Therefore,%
\[
\frac{\lambda_{n}n^{r/\alpha-1}}{\left(  n-1\right)  _{r-1}}\leq\frac{e\left(
G\right)  }{\frac{n}{r}\binom{n-1}{r-1}}+\frac{2r}{\alpha\left(
\alpha-1\right)  \log n}\leq\frac{ex\left(  \mathcal{P},n\right)  }{\binom
{n}{r}}+\frac{2r}{\alpha\left(  \alpha-1\right)  \log n}.
\]
Since $n$ can be arbitrary large, we can pass to limits obtaining
\[
\lambda^{\left(  \alpha\right)  }\left(  \mathcal{P}\right)  \leq\pi\left(
\mathcal{P}\right)  .
\]
This inequality together with (\ref{in1.1}) completes the proof of Theorem
\ref{th3}.
\end{proof}

\bigskip

\begin{proof}
[\textbf{Proof of Theorem \ref{th4}}]Since Proposition \ref{pro1} implies that
$\lambda^{\left(  1\right)  }\left(  \mathcal{P}\right)  \geq\pi\left(
\mathcal{P}\right)  ,$ to finish the proof we shall show that $\lambda
^{\left(  1\right)  }\left(  \mathcal{P}\right)  \leq\pi\left(  \mathcal{P}%
\right)  .$ We claim that if $\mathbf{x}=\left(  x_{1},\ldots,x_{n}\right)  $
is a nonnegative vector with $\left\Vert \mathbf{x}\right\Vert _{1}=1,$ then
\begin{equation}
P_{G}\left(  \mathbf{x}\right)  \leq\pi\left(  \mathcal{P}\right)
.\label{in16}%
\end{equation}
Because $P_{G}\left(  \mathbf{x}\right)  $ is continuous in each variable, it
suffices to prove the inequality for positive rational $x_{1},\ldots,x_{n}.$
Let thus
\[
x_{1}=k_{1}/p,\ldots,x_{n}=k_{n}/p,
\]
where $k_{1},\ldots,k_{n}$ are positive integers and $p$ is a common
denominator of $x_{1},\ldots,x_{n}$. Obviously the condition $x_{1}%
+\cdots+x_{n}=1$ implies that $p=k_{1}+\cdots+k_{n}.$ Therefore, (\ref{in16})
it is equivalent to
\begin{equation}
\frac{1}{\left(  k_{1}+\cdots+k_{n}\right)  ^{r}}P_{G}\left(  \left(
k_{1},\ldots,k_{n}\right)  \right)  \leq\pi\left(  \mathcal{P}\right)
.\label{in17}%
\end{equation}
Fix thus positive integers $k_{1},\ldots,k_{n}$ and note that%
\[
P_{G}\left(  \left(  k_{1},\ldots,k_{n}\right)  \right)  =P_{G\left(
k_{1},\ldots,k_{n}\right)  }\left(  \left(  1,\ldots,1\right)  \right)
=r!e\left(  G\left(  k_{1},\ldots,k_{n}\right)  \right)  .
\]
On the other hand, $v\left(  G\left(  k_{1},\ldots,k_{n}\right)  \right)
=k_{1}+\cdots+k_{n}$ and so
\[
e\left(  G\left(  k_{1},\ldots,k_{n}\right)  \right)  \leq\left(  \pi\left(
\mathcal{P}\right)  +o\left(  1\right)  \right)  \frac{\left(  k_{1}%
+\cdots+k_{n}\right)  ^{r}}{r!}.
\]
Here the term $o\left(  1\right)  $ tends to $0$ when $k_{1}+\cdots
+k_{n}\rightarrow\infty.$ Likewise, for every positive integer $L,$ we see
that
\begin{align*}
\frac{1}{\left(  k_{1}+\cdots+k_{n}\right)  ^{r}}P_{G}\left(  \left(
k_{1},\ldots,k_{n}\right)  \right)    & =\frac{1}{\left(  Lk_{1}+\cdots
+Lk_{n}\right)  ^{r}}P_{G}\left(  \left(  Lk_{1},\ldots,Lk_{n}\right)
\right)  \\
& =\frac{r!e\left(  G\left(  Lk_{1},\ldots,Lk_{n}\right)  \right)  }{\left(
Lk_{1}+\cdots+Lk_{n}\right)  ^{r}}\leq\pi\left(  \mathcal{P}\right)  +o\left(
1\right)  .
\end{align*}
Now, letting $L\rightarrow\infty,$ we obtain ((\ref{in17}), and so
$\lambda^{\left(  1\right)  }\left(  \mathcal{P}\right)  \leq\pi\left(
\mathcal{P}\right)  ,$ completing the proof of Theorem \ref{th4}.
\end{proof}

\bigskip

\begin{proof}
[\textbf{Proof of Theorem \ref{th5} }]Our proof follows an idea of Sidorenko
\cite{Sid87}, which he used in a similar setting. Let $\mathcal{P}$ be a
hereditary and multiplicative family of $r$-graphs. If $G\in\mathcal{P}_{n}$,
then for every integer $k\geq1,$ we have
\[
v\left(  G\left(  k,\ldots,k\right)  \right)  =kv\left(  G\right)  \text{
\ \ \ and \ \ }e\left(  G\left(  k,\ldots,k\right)  \right)  =k^{r}e\left(
G\right)  .
\]
Therefore,%
\[
e\left(  G\right)  =\frac{e\left(  G\left(  k,\ldots,k\right)  \right)
}{k^{r}}\leq\frac{ex\left(  \mathcal{P},nk\right)  }{k^{r}}.
\]
Since
\[
\lim_{k\rightarrow\infty}\frac{ex\left(  \mathcal{P},nk\right)  }{\binom
{nk}{r}}=\pi\left(  \mathcal{P}\right)  ,
\]
we see that%
\[
e\left(  G\right)  \leq\pi\left(  \mathcal{P}\right)  \binom{nk}{r}\frac
{1}{k^{r}}+o\left(  k\right)  \binom{nk}{r}\frac{1}{k^{r}}.
\]
Letting $k\rightarrow\infty,$ we obtain
\[
e\left(  G\right)  \leq\pi\left(  \mathcal{P}\right)  \frac{n^{r}}{r!},
\]
which proves (\ref{Seq1}).

To prove (\ref{upb}) note that by Proposition \ref{pro8},
\begin{align*}
\lambda^{\left(  \alpha\right)  }\left(  G\right)   &  =\lambda^{\left(
\alpha\right)  }\left(  G\left(  k,\ldots,k\right)  \right)  k^{r/\alpha
-r}\leq\lambda_{nk}^{\left(  \alpha\right)  }k^{r/\alpha-r}\\
&  =\left(  \lambda^{\left(  \alpha\right)  }\left(  \mathcal{P}\right)
+o\left(  1\right)  \right)  \left(  kn\right)  ^{r/\alpha-r}\frac
{k^{r/\alpha-r}}{\left(  kn\right)  ^{r/\alpha-1}}\\
&  =\left(  \lambda^{\left(  \alpha\right)  }\left(  \mathcal{P}\right)
+o\left(  1\right)  \right)  n^{r-r/\alpha}.
\end{align*}
Letting $k\rightarrow\infty,$ we see that
\[
\lambda^{\left(  \alpha\right)  }\left(  G\right)  \leq\lambda^{\left(
\alpha\right)  }\left(  \mathcal{P}\right)  n^{r-r/\alpha}=\pi\left(
\mathcal{P}\right)  n^{r-r/\alpha},
\]
completing the proof of Theorem \ref{th5}.
\end{proof}

\bigskip

\begin{proof}
[\textbf{Proof of Theorem \ref{th6}}]If $G\in\mathcal{P}_{n},$ Proposition
\ref{pro6} implies that
\[
\left(  \lambda^{\left(  \alpha\right)  }\left(  G\right)  \right)  ^{\alpha
}\leq\lambda^{\left(  1\right)  }\left(  G\right)  \left(  r!e\left(
G\right)  \right)  ^{\alpha-1}.
\]
Since $\mathcal{P}$ is hereditary and multiplicative, by Theorem \ref{th1} we
also have%
\[
\lambda^{\left(  1\right)  }\left(  G\right)  \leq\lambda^{\left(  1\right)
}\left(  \mathcal{P}\right)  =\pi\left(  \mathcal{P}\right)  ,
\]
and so,
\[
\lambda^{\left(  1\right)  }\left(  G\right)  \leq\pi\left(  \mathcal{P}%
\right)  ^{1/\alpha}\left(  r!e\left(  G\right)  \right)  ^{1-1/\alpha},
\]
completing the proof of Theorem \ref{th6}.
\end{proof}

\bigskip

\begin{proof}
[\textbf{Proof of Theorem \ref{thCon2}}]Let $\mathcal{P}$ be a monotone
property of $2$-graphs, let $H\notin\mathcal{P}$ be a graph with $\chi\left(
H\right)  =r=\min\left\{  \chi\left(  G\right)  :G\notin\mathcal{P}\right\}
,$ and let $K_{r}\left(  k,\ldots,k\right)  $ be the smallest regular
$r$-partite graph containing $H.$ First, obviously $\mathcal{P}\subset
Mon\left(  K_{r}\left(  k,\ldots,k\right)  \right)  ,$ and so, $\lambda
^{\left(  \alpha\right)  }\left(  \mathcal{P}\right)  \leq\lambda^{\left(
\alpha\right)  }\left(  Mon\left(  K_{r}\left(  k,\ldots,k\right)  \right)
\right)  ;$ however, Theorem \ref{thblo} implies that
\[
\lambda^{\left(  \alpha\right)  }\left(  \mathcal{P}\right)  \leq
\lambda^{\left(  \alpha\right)  }\left(  Mon\left(  K_{r}\left(
k,\ldots,k\right)  \right)  \right)  =\lambda^{\left(  \alpha\right)  }\left(
Mon\left(  K_{r}\right)  \right)  =\frac{r-2}{r-1}.
\]
Since $T_{r-1}\left(  n\right)  ,$ the $\left(  r-1\right)  $-partite
Tur\'{a}n graph of order $n,$ belongs to $\mathcal{P}$, we see that
\[
\lambda^{\left(  \alpha\right)  }\left(  \mathcal{P}\right)  =\frac{r-2}%
{r-1},
\]
proving Theorem \ref{thCon2}.
\end{proof}

\bigskip

\section{\label{CR}Concluding remarks}

We have started above a systematic study of the parameter $\lambda^{\left(
\alpha\right)  }\left(  G\right)  $ and its connections to extremal problems
for $r$-graphs. Similarly to eigenvalues of $2$-graphs, one may consider other
critical points of $P_{G},$ for instance, for every $r$-graph $G$ of order $n$
and every real number $\alpha\geq1,$ define
\[
\lambda_{\min}^{\left(  \alpha\right)  }\left(  G\right)  =\min_{\left\Vert
\mathbf{x}\right\Vert _{\alpha}=1}P_{G}\left(  \mathbf{x}\right)
=\min_{\left\vert x_{1}\right\vert ^{\alpha}+\cdots+\left\vert x_{n}%
\right\vert ^{\alpha}=1}r!\sum_{\left\{  i_{1},i_{2},\ldots,i_{r}\right\}  \in
E\left(  G\right)  }x_{i_{1}}x_{i_{2}}\cdots x_{i_{r}}.
\]
Obviously $\lambda_{\min}^{\left(  \alpha\right)  }\left(  G\right)  $ is
analogous to the smallest eigenvalue of $2$-graphs and one can come up with a
lot of supporting material about it, including a system of equations similar
to (\ref{eequ}). In particular, if $\mathcal{P}$ is a hereditary property of
$r$-graphs, we can define
\[
\lambda_{\min}^{\left(  \alpha\right)  }\left(  \mathcal{P},n\right)
=\min_{G\in\mathcal{P}_{n}}\lambda_{\min}^{\left(  \alpha\right)  }\left(
G\right)  ,
\]
and prove the following statement:

\begin{theorem}
Let $\alpha>1.$ If $\mathcal{P}$ is a hereditary property of $r$-graphs, then
the limit
\[
\lambda_{\min}^{\left(  \alpha\right)  }\left(  \mathcal{P}\right)
=\lim_{n\rightarrow\infty}\lambda_{\min}^{\left(  \alpha\right)  }\left(
\mathcal{P},n\right)  n^{r/\alpha-r}%
\]
exists.
\end{theorem}

Many obvious problems arise here, of which we mention the following
two:\medskip

\textbf{Problem 1} \emph{Let }$\alpha>1.$\emph{ For a }$2$\emph{-graph }%
$G$\emph{ of order }$n,$\emph{ study its \textquotedblleft}$\alpha
$\emph{-eigenvalues\textquotedblright, that is to say, critical values of
}$P_{G}$\emph{ over the unit sphere }$\left\vert x_{1}\right\vert ^{\alpha
}+\cdots+\left\vert x_{n}\right\vert ^{\alpha}=1.\medskip$

Since the smallest eigenvalue of $2$-graphs has proved to be a useful
structural parameter (see, e.g. \cite{Nik07}), one can investigate what role
plays $\lambda_{\min}^{\left(  \alpha\right)  }\left(  G\right)  $ for
$r$-graphs.\medskip

\textbf{Problem 2} \emph{Let }$\alpha>1.$\emph{ For an }$r$\emph{-graph }%
$G$\emph{ of order }$n,$\emph{ study which structural properties of }$G$
\emph{are} \emph{related to }$\lambda_{\min}^{\left(  \alpha\right)  }\left(
G\right)  .\medskip$

Although for $\alpha\neq1$ and $\alpha\neq r$ the parameter $\lambda^{\left(
\alpha\right)  }\left(  G\right)  $ is mostly auxiliary, it is challenging and
instructive to extend known results about $\lambda^{\left(  r\right)  }\left(
G\right)  $ to general $\alpha\geq1$. This is interesting even for $r=2.$
Keevash, Lenz and Mubayi have pointed to such generalization in their
Corollary 2, but a lot more work is pending. An important initial endeavour
would be to recover parts of the Perron-Frobenius theory for $\lambda^{\left(
\alpha\right)  }\left(  G\right)  .$ In particular, what is the set of all
critical vectors corresponding to $\lambda^{\left(  \alpha\right)  }\left(
G\right)  $.\medskip

\textbf{Problem 3 }\emph{Given an }$r$\emph{-graph }$G$\emph{ of order }%
$n,$\emph{ determine the set of all }$n$\emph{-vectors }$x$\emph{ with
}$\left\Vert \mathbf{x}\right\Vert _{\alpha}=1$\emph{ and} $\lambda^{\left(
\alpha\right)  }\left(  G\right)  =P_{G}\left(  \mathbf{x}\right)  .\medskip$

In view of the importance of flat properties, the following problem is
natural, although probably quite difficult:\medskip

\textbf{Problem 4 }\emph{Characterize all flat properties of }$r$%
\emph{-graphs.}\medskip

A particular case of this problem arises in connection to Theorem \ref{th9}.
It is curious how rich can be properties of graphs defined by forbidden
induced blowups of covering graphs.  \medskip

\textbf{Problem 5 }\emph{For which non-covering graphs F, the property
Her(\{F\}) can be represented as Her(X) for some family X of blowups of
covering graphs:?}\bigskip

\textbf{Acknowledgement}\ Much of the material presented above was conceived
in the Spring of 2012, when the author read a course on spectra of hypergraphs
at the Math. Department of the University of Memphis. The author is grateful
to Peter Keevash, John Lenz and Dhruv Mubayi for sharing their manuscript
\cite{KLM13}, which greatly motivated the writing of the present paper.

Also, the author is much indebted to Alex Sidorenko for useful discussions.
His comments on an earlier version of the paper led to considerable
improvement of the presentation.

\end{document}